%% file: DisjunctionsOnLorentzCone.tex
\documentclass[12pt]{article}
\usepackage{amsmath,amssymb,amsthm,mathtools}
\usepackage{url,hyperref,setspace,color}
\usepackage{subfigure,graphicx,epsfig,caption}
\usepackage{enumerate}

\newtheorem{lemma}{Lemma}
\newtheorem{proposition}{Proposition}
\newtheorem{theorem}{Theorem}
\newtheorem{corollary}{Corollary}
\newtheorem{assumption}{Assumption}
\newtheorem*{remark}{Remark}

\newcommand{\norm}[1]{\left\Vert #1 \right\Vert}

\newcommand{\capH}{H}
\newcommand{\capB}{B}
\newcommand{\capM}{M}
\newcommand{\capD}{D}
\newcommand{\cR}{\mathcal{R}}
\newcommand{\cP}{\mathcal{P}}
\newcommand{\cQ}{\mathcal{Q}}
\newcommand{\cN}{\mathcal{N}}
\newcommand{\cM}{\mathcal{M}}
\newcommand{\cA}{\mathcal{A}}
\newcommand{\cB}{\mathcal{B}}

\newcommand{\R}{\mathbb{R}}
\newcommand{\Z}{\mathbb{Z}}
\newcommand{\N}{\mathbb{N}}
\newcommand{\K}{\mathbb{K}}

\DeclareMathOperator{\spann}{span}
\DeclareMathOperator{\conv}{conv}

\DeclareMathOperator{\clconv}{\overline{conv}}

\DeclareMathOperator{\intt}{int}

\DeclareMathOperator{\bd}{bd}

\DeclareMathOperator{\rec}{rec}

\DeclareMathOperator{\sign}{sign}

\begin{document}

\title{Two-Term Disjunctions on the Second-Order Cone}
\author{Fatma K{\i}l{\i}n\c{c}-Karzan\thanks{Tepper School of Business, Carnegie Mellon University, Pittsburgh, PA, {\tt fkilinc@andrew.cmu.edu}.}
\and Sercan Y{\i}ld{\i}z \thanks{Tepper School of Business, Carnegie Mellon University, Pittsburgh, PA, {\tt syildiz@andrew.cmu.edu}.}
}

\date{\today}

\maketitle

\begin{abstract}
Balas introduced disjunctive cuts in the 1970s for mixed-integer linear programs. Several recent papers have attempted to extend this work to mixed-integer conic programs. In this paper we study the structure of the convex hull of a two-term disjunction applied to the second-order cone, and develop a methodology to derive closed-form expressions for convex inequalities describing the resulting convex hull. Our approach is based on first characterizing the structure of undominated valid linear inequalities for the disjunction and then using conic duality to derive a family of convex, possibly nonlinear, valid inequalities that correspond to these linear inequalities. We identify and study the cases where these valid inequalities can equivalently be expressed in conic quadratic form and where a single inequality from this family is sufficient to describe the convex hull. In particular, our results on two-term disjunctions on the second-order cone generalize related results on split cuts by Modaresi, K{\i}l{\i}n\c{c}, and Vielma, and by Andersen and Jensen.

\vspace{3mm}


\end{abstract}

\section{Introduction}
A mixed-integer conic program is a problem of the form
\begin{equation*}
\sup\{d^\top x:\;Ax=b,\;x\in\K,\;x_j\in\Z\;\forall j\in J\}
\end{equation*}
where $\K\subset\R^n$ is a regular (full-dimensional, closed, convex, and pointed) cone, $A$ is an $m\times n$ real matrix, $d$ and $b$ are real vectors of appropriate dimensions, and $J\subseteq\{1,\ldots,n\}$. Mixed-integer conic programming (MICP) models arise naturally as robust versions of mixed-integer linear programming (MILP) models in finance, management, and engineering \cite{BS2013b,Burer_Letchford_12}.
MILP is the special case of MICP where $\K$ is the nonnegative orthant, and it has itself numerous applications. A successful approach to solving MILP problems has been to first solve the continuous relaxation, then add cuts, and finally perform branch-and-bound using this strengthened formulation. A powerful way of generating such cuts is to impose a valid disjunction on the continuous relaxation and derive tight convex inequalities for the resulting disjunctive set. Such inequalities are known as \emph{disjunctive cuts}. Specifically, the integrality conditions on the variables $x_j$, $j\in J$, imply linear \emph{split disjunctions} of the form $\pi^\top x\leq\pi_0\,\vee\,\pi^\top x\geq\pi_0+1$ where $\pi_0\in\Z$, $\pi_j\in\Z$, $j\in J$, and $\pi_j=0$, $j\not\in J$. Following this approach, the feasible region for MICP problems can be relaxed to $\{x\in\K:\,Ax=b,\pi^\top x\leq\pi_0\vee\pi^\top x\geq\pi_0+1\}$. More general two-term disjunctions arise in complementarity \cite{JSRF2006,Tawarmalani_Richard_Chung_10} and other non-convex optimization \cite{B2012,Burer_Saxena_12} problems. Therefore, it is interesting to study relaxations of MICP problems of the form
\begin{gather}
\sup\{d^\top x:\;x\in C_1\cup C_2\}\quad\text{where}\notag\\
C_i:=\{x\in\K:\;Ax=b,\;c_i^\top x\geq c_{i,0}\}\quad\text{for}\quad i\in\{1,2\}\label{eq:DisjunctiveSet}
\end{gather}
and to derive strong valid inequalities for the convex hull $\conv(C_1\cup C_2)$, or the closed convex hull $\clconv(C_1\cup C_2)$. When $\K$ is the nonnegative orthant, Bonami et al. \cite{BCCMZ2013} characterize $\clconv(C_1\cup C_2)$ by a finite set of linear inequalities. The purpose of this paper is to study the structure of $\clconv(C_1\cup C_2)$ for other cones such as the second-order (Lorentz) cone $\K_2^n:=\{x\in\R^n:\,\norm{(x_1;\ldots;x_{n-1})}_2\leq x_n\}$, or more generally the $p$-order cone $\K^n_p:=\{x\in\R^n:\,\norm{(x_1;\ldots;x_{n-1})}_p\leq x_n\}$ where $p\in(1,\infty)$, and provide the explicit description of $\clconv(C_1\cup C_2)$ with convex inequalities in the space of the original variables. We first review related results from the literature.

Disjunctive cuts were introduced by Balas \cite{B1971} for MILP in the early 1970s. Since then, disjunctive cuts have been studied extensively in mixed integer linear and nonlinear optimization \cite{Balas_79,Sherali_Shetti_80,Balas_Ceria_Cornuejols_93,Cornuejols_Lemarechal_06,Saxena_Bonami_Lee_08,Cadoux_10,Kilinc_Linderoth_Luedtke_10,Burer_Saxena_12}. \emph{Chv{\'a}tal-Gomory}, \emph{lift-and-project}, \emph{mixed-integer rounding (MIR)}, and \emph{split cuts} are all special types of disjunctive cuts. Recent efforts on extending the cutting plane theory for MILP to the MICP setting include the work of \c{C}ezik and Iyengar \cite{CI2005} for Chvatal-Gomory cuts, Stubbs and Mehrotra \cite{SM1999}, Drewes \cite{D2009}, Drewes and Pokutta \cite{Drewes_Pokutta_10}, and Bonami \cite{B2011} for lift-and-project cuts, and Atamt\"{u}rk and Narayanan \cite{AV2010,Atamturk_Narayanan_11} for MIR cuts. K{\i}l{\i}n\c{c}-Karzan \cite{KK} analyzed properties of minimal valid linear inequalities for general conic sets with a disjunctive structure and showed that these are sufficient to describe the closed convex hull. Such general sets from \cite{KK} include two-term disjunctions on the cone $\K$ considered in this paper. Bienstock and Michalka \cite{BM} studied the characterization and separation of valid linear inequalities that convexify the epigraph of a convex, differentiable function restricted to a non-convex domain. In the last few years, there has been growing interest in developing closed-form expressions for convex inequalities that fully describe the convex hull of a disjunctive conic set. Dadush et al. \cite{DDV2011} and Andersen and Jensen \cite{AJ2013} derived split cuts for ellipsoids and the second-order cone, respectively. Modaresi et al. \cite{MKV} extended this work on split disjunctions to essentially all cross-sections of the second-order cone, and studied their theoretical and computational relations with extended formulations and conic MIR inequalities in \cite{MKV2}. Belotti et al. \cite{Belotti_Goez_Polik_Terlaky_12} studied the families of quadratic surfaces having fixed intersections with two given hyperplanes and showed that these families can be described by a single parameter. Also, in \cite{BGPRT}, they  identified a procedure for constructing two-term disjunctive cuts under the assumptions that $C_1\cap C_2=\emptyset$ and the sets $\{x\in\K:\,Ax=b,c_1^\top x=c_{1,0}\}$ and $\{x\in\K:\,Ax=b,c_2^\top x=c_{2,0}\}$ are bounded.

In this paper we study general two-term disjunctions on conic sets and give closed-form expressions for the tightest disjunctive cuts that can be obtained from these disjunctions in a large class of instances. We focus on the case where $C_1$ and $C_2$ in \eqref{eq:DisjunctiveSet} above have an empty set of equations $Ax=b$. That is to say, we consider
\begin{equation}\label{eq:DisjSet}
C_1:=\{x\in\K:\,c_1^\top x\geq c_{1,0}\}\;\;\;\text{and}\;\;\;
C_2:=\{x\in\K:\,c_2^\top x\geq c_{2,0}\}.
\end{equation}
We note, however, that our results can easily be extended to two-term disjunctions on sets of the form $\{x\in\R^n:\,Ax-b\in\K\}$, where $A$ has full row rank, through the affine transformation discussed in \cite{AJ2013}.
Our main contribution is to give an explicit outer description of $\clconv(C_1\cup C_2)$ when $\K$ is the second-order cone. Similar results have previously appeared in \cite{AJ2013}, \cite{MKV}, and \cite{BGPRT}. Nevertheless, our work is set apart from \cite{AJ2013} and \cite{MKV} by the fact that we study two-term disjunctions on the cone $\K$ in their full generality and do not restrict our attention to split disjunctions, which are defined by {\em parallel} hyperplanes. Furthermore, unlike \cite{BGPRT}, we do not assume that $C_1\cap C_2=\emptyset$ and the sets $\{x\in\K:\,c_1^\top x=c_{1,0}\}$ and $\{x\in\K:\,c_2^\top x=c_{2,0}\}$ are bounded. In the absence of such assumptions, the resulting convex hulls turn out to be significantly more complex in our general setting. We also stress that our proof techniques originate from a conic duality perspective and are completely different from what is employed in the aforementioned papers; in particular, we believe that they are more intuitive in terms of their derivation, and more transparent in understanding the structure of the resulting convex hulls. Therefore, we hope that they have the potential to be instrumental in extending several important existing results in this growing area of research.

The remainder of this paper is organized as follows. Section~\ref{sec:preliminaries} introduces the tools that will be useful to us in our analysis. In Section~\ref{sec:sub:notation} we set out our notation and basic assumptions. In Section~\ref{sec:sub:VLIProperties} we characterize the structure of undominated valid linear inequalities describing $\clconv(C_1\cup C_2)$ when $\K$ is a regular cone. In Section~\ref{sec:SecondOrder} we focus on the case where $\K$ is the second-order cone. In Section~\ref{sec:sub:CutDerivation} we state and prove our main result, Theorem~\ref{thm:main}. The proof uses conic duality, along with the characterization from Section~\ref{sec:sub:VLIProperties}, to derive a family of convex, possibly linear or conic, valid inequalities \eqref{eq:main} for $\clconv(C_1\cup C_2)$. In Sections~\ref{sec:sub:ConicQuadratic} and \ref{sec:SingleConvex}, we identify and study the cases where these inequalities can equivalently be expressed in conic quadratic form and where only one inequality of the form \eqref{eq:main} is sufficient to describe $\clconv(C_1\cup C_2)$. Our results imply in particular that a single conic valid inequality is always sufficient for split disjunctions. Nevertheless, there are cases where it is not possible to obtain $\clconv(C_1\cup C_2)$ with a single inequality of the form \eqref{eq:main}. In Section~\ref{sec:MultipleIneqs} we study those cases and outline a technique to characterize $\clconv(C_1\cup C_2)$ with closed-form formulas.
While the majority of our work is geared towards the second-order cone, in Section~\ref{sec:p-order} we look at the $n$-dimensional $p$-order cone $\K_p^n$ with $p\in(1,\infty)$ and study elementary split disjunctions on this set. That is, we consider sets $C_1$ and $C_2$ defined as in \eqref{eq:DisjSet} where $c_1$ and $c_2$ are multiples of the $i^{\text{th}}$ standard unit vector, $e^i$, $i\in\{1,\ldots,n-1\}$. We show that one can obtain a single conic inequality that describes the convex hull using our framework in this setup. This provides an alternative proof of a similar result by Modaresi et al. \cite{MKV}.

\section{Preliminaries}\label{sec:preliminaries}

The main purpose of this section is to characterize the structure of undominated valid linear inequalities for $\clconv(C_1\cup C_2)$ when $\K$ is a regular cone and $C_1$ and $C_2$ are defined as in \eqref{eq:DisjSet}. First, we present our notation and assumptions.

\subsection{Notation and Assumptions}\label{sec:sub:notation}

Given a set $S\subseteq\R^n$, we let $\spann S$, $\intt S$, and $\bd S$ denote the linear span, interior, and boundary of $S$, respectively. We use $\rec S$ to refer to the recession cone of a convex set $S$. The {\em dual cone} of $\K\subseteq\R^n$ is $\K^*:=\{y\in\R^n:\,y^\top x\geq 0\,\forall x\in\K\}$. Recall that the dual cone $\K^*$ of a regular cone $\K$ is also regular and the dual of $\K^*$ is $\K$ itself.

We can always scale the inequalities $c_1^\top x\geq c_{1,0}$ and $c_2^\top x\geq c_{2,0}$ defining the disjunction so that their right-hand sides are $0$ or $\pm 1$. Therefore, from now on we assume that $c_{1,0},c_{2,0}\in\{0,\pm 1\}$ for notational convenience.

When $C_1\subseteq C_2$, we have $\clconv(C_1\cup C_2)=C_2$. Similarly, when $C_1\supseteq C_2$, we have $\clconv(C_1\cup C_2)=C_1$. In the remainder we focus on the case where $C_1\not\subseteq C_2$ and $C_1\not\supseteq C_2$.

\begin{assumption}\label{As:A1}
$C_1\not\subseteq C_2$ and $C_1\not\supseteq C_2$.
\end{assumption}

In particular, Assumption~\ref{As:A1} implies $c_i\not\in-\K^*$ when $c_{i,0}=+1$ and $c_i\not\in\K^*$ when $c_{i,0}=-1$. We also need the following technical assumption in our analysis.

\begin{assumption}\label{As:A2}
$C_1$ and $C_2$ are strictly feasible sets. That is, $C_1\cap\intt\K\neq\emptyset$ and $C_2\cap\intt\K\neq\emptyset$.
\end{assumption}

The set $C_i$ is always strictly feasible when it is nonempty and $c_{i,0}\in\{\pm 1\}$. Therefore, we need Assumption~\ref{As:A2} to supplement Assumption~\ref{As:A1} only when $c_{1,0}=0$ or $c_{2,0}=0$. Note that, under Assumption~\ref{As:A2}, the sets $C_1$ and $C_2$ always have nonempty interior. Assumptions~\ref{As:A1} and \ref{As:A2} have several simple implications, which we state next. The first lemma extends ideas from Balas \cite{B1974} to disjunctions on more general convex sets. Its proof is left to the appendix.

\begin{lemma}\label{lem:closure}
Let $S\subset\R^n$ be a closed, convex, pointed set, $S_1:=\{x\in S:\,c_1^\top x\geq c_{1,0}\}$, and $S_2:=\{x\in S:\,c_2^\top x\geq c_{2,0}\}$ for $c_1,c_2\in\R^n$ and $c_{1,0},c_{2,0}\in\R$. Suppose $S_1\not\subseteq S_2$ and $S_1\not\supseteq S_2$. Then
\begin{enumerate}[(i)]
\item $S_1\cup S_2$ is not convex unless $S_1\cup S_2=S$,
\item $\clconv(S_1\cup S_2)=\conv(S_1^+\cup S_2^+)$ where $S_1^+:=S_1+\rec S_2$ and $S_2^+:=S_2+\rec S_1$.
\end{enumerate}
\end{lemma}

Clearly, when $\clconv(C_1\cup C_2)\neq\K$, we do not need to derive any new inequalities to get a description of the closed convex hull. The next lemma obtains a natural consequence of Assumption~\ref{As:A1} through conic duality.

\begin{lemma}\label{lem:inconsistent}
Consider $C_1, C_2$ defined as in \eqref{eq:DisjSet}. Suppose Assumption~\ref{As:A1} holds. Then the following system of inequalities in the variable $\beta$ is inconsistent:
\begin{equation}\label{eq:inconsistent1}
\beta\geq 0,\quad\beta c_{1,0}\geq c_{2,0},\quad c_2-\beta c_1\in\K^*.
\end{equation}
Similarly, the following system of inequalities in the variable $\beta$ is inconsistent:
\begin{equation}\label{eq:inconsistent2}
\beta\geq 0,\quad\beta c_{2,0}\geq c_{1,0},\quad c_1-\beta c_2\in\K^*.
\end{equation}
\end{lemma}

\begin{proof}
Suppose there exists $\beta^*$ satisfying \eqref{eq:inconsistent1}. For all $x\in\K$, this implies $(c_2-\beta^*c_1)^\top x\geq 0\geq c_{2,0}-\beta^*c_{1,0}$. Then any point $x\in C_1$ satisfies $\beta^*c_1^\top x\geq\beta^*c_{1,0}$ and therefore, $c_2^\top x\geq c_{2,0}$. Hence, $C_1\subseteq C_2$ which contradicts Assumption~\ref{As:A1}. The proof for the inconsistency of \eqref{eq:inconsistent2} is similar.
\end{proof}

\subsection{Properties of Undominated Valid Linear Inequalities}\label{sec:sub:VLIProperties}

A valid linear inequality $\mu^\top x\geq\mu_0$ for a feasible set $S\subseteq\K$ is said to be {\em tight} if $\inf_x\{\mu^\top x:\,x\in S\}=\mu_0$ and {\em strongly tight} if there exists $x^*\in S$ such that $\mu^\top x^*=\mu_0$.

A valid linear inequality $\nu^\top x\geq\nu_0$ for a strictly feasible set $S\subseteq\K$ is said to {\em dominate} another valid linear inequality $\mu^\top x\geq\mu_0$ if it is not a positive multiple of $\mu^\top x\geq\mu_0$ and implies $\mu^\top x\geq\mu_0$ together with the cone constraint $x\in\K$. Furthermore, a valid linear inequality $\mu^\top x\geq\mu_0$ is said to be {\em undominated} if there does not exists another valid linear inequality $\nu^\top x\geq\nu_0$ such that $(\mu-\nu,\mu_0-\nu_0)\in\K^*\times-\R_+\setminus\{(0,0)\}$. This notion of domination is closely tied with the $\K$-minimality definition of \cite{KK} which says that a valid linear inequality $\mu^\top x\geq\mu_0$ is {\em $\K$-minimal} if there does not exist another valid linear inequality $\nu^\top x\geq\nu_0$ such that $(\mu-\nu,\mu_0-\nu_0)\in(\K^*\setminus\{0\})\times-\R_+$. In particular, a valid linear inequality for $\clconv(C_1\cup C_2)$ is undominated in the sense considered here if and only if it is $\K$-minimal and tight on $\clconv(C_1\cup C_2)$. In \cite{KK}, $\K$-minimal inequalities are defined and studied for sets of the form
\begin{equation*}
\left\{x\in\R^n:\,Ax\in\capH,\,x\in\K\right\},
\end{equation*}
where $\capH$ is an arbitrary set and $\K$ is a regular cone. Our set $C_1\cup C_2$ can be represented in the form above as
\begin{equation*}
\left\{x\in\R^n:\,\begin{pmatrix}c_1^T \\ c_2^T\end{pmatrix} x=\left\{{\{c_{1,0}\}+\R_+\choose\R}\right\} \bigcup \left\{{\R\choose \{c_{2,0}\}+\R_+}\right\},~x\in\K \right\}.
\end{equation*}
Because $C_1\cup C_2$ is full-dimensional under Assumption~\ref{As:A2}, Proposition 1 of \cite{KK} can be used to conclude that the extreme rays of the convex cone of valid linear inequalities
\begin{equation*}
\capM:=\left\{(\mu,\mu_0)\in\R^n\times\R:\;\mu^\top x\geq\mu_0\;\forall x\in\clconv(C_1\cup C_2)\right\}
\end{equation*}
are either tight, $\K$-minimal inequalities or implied by the cone constraint $x\in\K$. Hence, one needs to add only undominated valid linear inequalities to the cone constraint $x\in\K$ to obtain an outer description of $\clconv(C_1\cup C_2)$.

Because $C_1$ and $C_2$ are strictly feasible sets by Assumption~\ref{As:A2}, conic duality implies that a linear inequality $\mu^\top x\geq\mu_0$ is valid for $\clconv(C_1\cup C_2)$ if and only if there exist $\alpha_1,\alpha_2,\beta_1,\beta_2$ such that $(\mu,\mu_0,\alpha_1,\alpha_2,\beta_1,\beta_2)$ satisfies
\begin{equation}\label{eq:VLI}
\begin{gathered}
\mu=\alpha_1+\beta_1c_1,\\
\mu=\alpha_2+\beta_2c_2,\\
\beta_1 c_{1,0}\geq\mu_0,\quad\beta_2 c_{2,0}\geq\mu_0,\\
\alpha_1,\alpha_2\in\K^*,\quad\beta_1,\beta_2\in\R_+.
\end{gathered}
\end{equation}
This system can be reduced slightly when we consider {\em undominated} valid linear inequalities.

\begin{proposition}\label{prop:VLIReduced}
Consider $C_1, C_2$ defined as in \eqref{eq:DisjSet} with $c_{1,0},c_{2,0}\in\{0,\pm 1\}$. Suppose Assumptions~\ref{As:A1} and \ref{As:A2} hold.
Then, up to positive scaling, any undominated valid linear inequality for $\clconv(C_1\cup C_2)$ has the form $\mu^\top x\geq\min\{c_{1,0},c_{2,0}\}$ with $(\mu,\alpha_1,\alpha_2,\beta_1,\beta_2)$ satisfying
\begin{equation}\label{eq:VLIReduced}
\begin{gathered}
\mu=\alpha_1+\beta_1c_1,\\
\mu=\alpha_2+\beta_2c_2,\\
\min\{\beta_1 c_{1,0},\beta_2 c_{2,0}\}=\min\{c_{1,0},c_{2,0}\},\\
\alpha_1,\alpha_2\in\bd\K^*,\quad\beta_1,\beta_2\in\R_+\setminus\{0\}.
\end{gathered}
\end{equation}
\end{proposition}

\begin{proof}
Let $\nu^\top x\geq\nu_0$ be a valid inequality for $\clconv(C_1\cup C_2)$. Then there exist $\alpha_1,\alpha_2,\beta_1,\beta_2$ such that $(\nu,\nu_0,\alpha_1,\alpha_2,\beta_1,\beta_2)$ satisfies \eqref{eq:VLI}. If $\beta_1=0$ or $\beta_2=0$, then $\nu^\top x\geq\nu_0$ is implied by the cone constraint $x\in\K$. If $\min\{\beta_1 c_{1,0},\beta_2 c_{2,0}\}>\nu_0$, then $\nu^\top x\geq\nu_0$ is implied by the valid inequality $\nu^\top x\geq\min\{\beta_1 c_{1,0},\beta_2 c_{2,0}\}$. Hence, we can assume without any loss of generality that any undominated valid linear inequality for $\clconv(C_1\cup C_2)$ has the form $\nu^\top x\geq\nu_0$ with $(\nu,\nu_0,\alpha_1,\allowbreak\alpha_2,\beta_1,\beta_2)$ satisfying
\begin{equation*}
\begin{gathered}
\nu=\alpha_1+\beta_1c_1,\\
\nu=\alpha_2+\beta_2c_2,\\
\min\{\beta_1 c_{1,0},\beta_2 c_{2,0}\}=\nu_0\\
\alpha_1,\alpha_2\in\K^*,\;\beta_1,\beta_2\in\R_+\setminus\{0\}.
\end{gathered}
\end{equation*}

We are now going to show that when $\alpha_1\in\intt\K^*$ or $\alpha_2\in\intt\K^*$, any such inequality is either dominated or equivalent to a valid inequality $\mu^\top x\geq\min\{c_{1,0},c_{2,0}\}$ that satisfies \eqref{eq:VLIReduced}. Assume without any loss of generality that $\alpha_2\in\intt\K^*$. There are two cases that we need to consider: $\alpha_1=0$ and $\alpha_1\neq 0$.

First suppose $\alpha_1=0$. We have $\alpha_2=\beta_1c_1-\beta_2c_2\in\intt\K^*$. By Lemma~\ref{lem:inconsistent} and taking $\beta_1,\beta_2>0$ into account, we conclude $\beta_2c_{2,0}<\beta_1c_{1,0}$. Hence, $\nu_0=\beta_2c_{2,0}$. If $\nu_0>0$, let $0<\epsilon'<\beta_1$ be such that $\alpha_2':=\alpha_2-\epsilon'c_1\in\K^*$ and $\beta_2c_{2,0}\leq\beta_1c_{1,0}-\epsilon'c_{1,0}$ and define $\beta_1':=\beta_1-\epsilon'$ and $\mu:=\nu-\epsilon'c_1$. If $\nu_0\leq 0$, let $\epsilon'>0$ be such that $\alpha_2':=\alpha_2+\epsilon'c_1\in\K^*$ and $\beta_2c_{2,0}\leq\beta_1c_{1,0}+\epsilon'c_{1,0}$ and define $\beta_1':=\beta_1+\epsilon'$ and $\mu:=\nu+\epsilon'c_1$. In either case, the inequality $\mu^\top x\geq\nu_0$ is valid for $\clconv(C_1\cup C_2)$ because $(\mu,\nu_0,\alpha_1,\alpha_2',\beta_1',\beta_2)$ satisfies \eqref{eq:VLI}. Furthermore, it dominates (or in the case of $\nu_0=0$, is equivalent to) $\nu^\top x\geq\nu_0$ because $\mu=\frac{\beta_1'}{\beta_1}\nu$ and $\beta_1'<\beta_1$ when $\nu_0>0$ and $\beta_1'>\beta_1$ when $\nu_0\leq 0$.

Now suppose $\alpha_1\neq 0$. Let $0<\epsilon''\leq 1$ be such that $\alpha_2'':=\alpha_2-\epsilon''\alpha_1\in\bd\K^*$, and define $\alpha_1'':=(1-\epsilon'')\alpha_1$ and $\mu:=\nu-\epsilon''\alpha_1$. The inequality $\mu^\top x\geq\nu_0$ is valid for $\clconv(C_1\cup C_2)$ because $(\mu,\nu_0,\alpha_1'',\alpha_2'',\beta_1,\beta_2)$ satisfies \eqref{eq:VLI}. Furthermore, $\mu^\top x\geq\nu_0$ dominates $\nu^\top x\geq\nu_0$ since $\nu-\mu=\epsilon''\alpha_1\in\K^*\setminus\{0\}$.

Finally, note that we can scale any valid inequality $\mu^\top x\geq\nu_0$ (along with the tuple $(\mu,\nu_0,\alpha_1,\alpha_2,\beta_1,\beta_2)$) so that $\nu_0\in\{0,\pm 1\}$. Using the fact that $\beta_1,\beta_2>0$ in an undominated valid inequality, we arrive at
\begin{align*}
\nu_0&=\sign(\nu_0)=\sign(\min\{\beta_1 c_{1,0},\beta_2 c_{2,0}\})\\
&=\min\{\underbrace{\sign(\beta_1 c_{1,0})}_{=\sign(c_{1,0})},\underbrace{\sign(\beta_2 c_{2,0})}_{=\sign(c_{2,0})}\}=\min\{c_{1,0},c_{2,0}\}.
\end{align*}
\end{proof}

\begin{remark}\label{rem:VLIReduced}
Under the assumptions of Proposition~\ref{prop:VLIReduced}, in an undominated valid linear inequality $\mu^\top x\geq\min\{c_{1,0},c_{2,0}\}$, we can assume that at least one of $\beta_1$ and $\beta_2$ is equal to $1$ in \eqref{eq:VLIReduced} without any loss of generality. In particular,
\begin{enumerate}[(i)]
\item if $c_{1,0}>c_{2,0}$, we can assume that $\beta_2=1$, $\beta_1c_{1,0}\geq c_{2,0}$, and $\beta_1c_1-c_2\notin\pm\intt\K^*$ holds,
\item if $c_{1,0}=c_{2,0}$, we can assume that either $\beta_2=1$, $\beta_1c_{1,0}\geq c_{2,0}$, and $\beta_1c_1-c_2\notin\pm\intt\K^*$ or
$\beta_1=1$, $\beta_2c_{2,0}\geq c_{1,0}$, and $\beta_2c_2-c_1\notin\pm\intt\K^*$ holds.
\end{enumerate}
\end{remark}

\begin{proof}
The remark follows from a careful look at the proof of Proposition~\ref{prop:VLIReduced}.

First suppose $c_{1,0}>c_{2,0}$. In this case the equality $\min\{\beta_1 c_{1,0},\beta_2 c_{2,0}\}=\min\{c_{1,0},c_{2,0}\}$ reduces to $\beta_2 c_{2,0}=c_{2,0}$ since $\sign(\beta_1 c_{1,0})=c_{1,0}>c_{2,0}=\sign(\beta_2 c_{2,0})$. This already implies $\beta_2=1$ when $c_{2,0}\in\{\pm 1\}$. When $c_{2,0}=0$, any undominated valid linear inequality has the form $\mu^\top x\geq 0$ and we can scale this inequality (along with the tuple $(\mu,\alpha_1,\alpha_2,\beta_1,\beta_2)$ that satisfies \eqref{eq:VLIReduced}) by a positive scalar to obtain an equivalent valid inequality with $\beta_2=1$. Therefore, when $c_{1,0}>c_{2,0}$, any undominated valid linear inequality for $\clconv(C_1\cup C_2)$ has the form $\mu^\top x\geq c_{2,0}$ with $(\mu,\alpha_1,\alpha_2,\beta)$ satisfying the system
\begin{equation*}
\begin{gathered}
\mu=\alpha_1+\beta c_1,\\
\mu=\alpha_2+c_2,\\
\alpha_1,\alpha_2\in\bd\K^*,\;\beta\in\R_+\setminus\{0\}.
\end{gathered}
\end{equation*}
In particular, this implies $c_2-\beta c_1\notin\intt\K^*$ since we must have $\alpha_1=\alpha_2+(c_2-\beta c_1)\in\bd\K^*$ and $c_2-\beta c_1\notin-\intt\K^*$ since we must have $\alpha_2=\alpha_1-(c_2-\beta c_1)\in\bd\K^*$.

Now suppose $c_{1,0}=c_{2,0}$. In this case the equality $\min\{\beta_1 c_{1,0},\beta_2 c_{2,0}\}=\min\{c_{1,0},c_{2,0}\}$ becomes $\min\{\beta_1 c_{1,0},\beta_2 c_{2,0}\}=c_{1,0}=c_{2,0}$. When $c_{1,0}=c_{2,0}\in\{\pm 1\}$, this implies either $\beta_1=1$ or $\beta_2=1$. Otherwise, any undominated valid linear inequality has the form $\mu^\top x\geq 0$ and we can again scale this inequality (along with the tuple $(\mu,\alpha_1,\alpha_2,\beta_1,\beta_2)$ that satisfies \eqref{eq:VLIReduced}) by a positive scalar to make, say, $\beta_2$ equal to 1. Therefore, when $c_{1,0}=c_{2,0}$, any undominated valid linear inequality for $\clconv(C_1\cup C_2)$ has the form $\mu^\top x\geq c_{1,0}=c_{2,0}$ with $(\mu,\alpha_1,\alpha_2,\beta)$ satisfying one of the following systems:
\begin{equation*}
\begin{aligned}[c]
(i)
\end{aligned}
\quad
\begin{aligned}[l]
&\mu=\alpha_1+\beta c_1,\\
&\mu=\alpha_2+c_2,\\
&\beta c_{1,0}\geq c_{1,0},\\
&\alpha_1,\alpha_2\in\bd\K^*,\;\beta\in\R_+\setminus\{0\},
\end{aligned}
\qquad
\begin{aligned}[c]
(ii)
\end{aligned}
\quad
\begin{aligned}[l]
&\mu=\alpha_1+c_1,\\
&\mu=\alpha_2+\beta c_2,\\
&\beta c_{2,0}\geq c_{2,0},\\
&\alpha_1,\alpha_2\in\bd\K^*,\;\beta\in\R_+\setminus\{0\}.
\end{aligned}
\end{equation*}
In case $(i)$ this implies $c_2-\beta c_1\notin\pm\intt\K^*$. In $(ii)$ this implies $c_1-\beta c_2\notin\pm\intt\K^*$.
\end{proof}

\section{Deriving the Disjunctive Cut}\label{sec:SecondOrder}

In this section we focus on the case where $\K$ is the second-order cone $\K_2^n:=\{x\in\R^n:\,\norm{\tilde{x}}_2\leq x_n\}$ and $\tilde{x}:=(x_1;\ldots;x_{n-1})$. Recall that the dual cone of $\K_2^n$ is again $\K_2^n$.

As in the previous section, we consider $C_1$ and $C_2$ defined as in \eqref{eq:DisjSet} with $c_{1,0},c_{2,0}\in\{0,\pm 1\}$ and suppose that Assumptions~\ref{As:A1} and \ref{As:A2} hold. We also assume without any loss of generality that $c_{1,0}\geq c_{2,0}$. Sets $C_1$ and $C_2$ that satisfy these conditions are said to satisfy the {\em basic disjunctive setup}. When in addition $\K=\K_2^n$, the sets $C_1$ and $C_2$ are said to satisfy the {\em second-order cone disjunctive setup}.

\subsection{A Convex Valid Inequality}\label{sec:sub:CutDerivation}

Proposition~\ref{prop:VLIReduced} gives a nice characterization of the form of undominated linear inequalities valid for $\clconv(C_1\cup C_2)$. In the following we use this characterization and show that, for a given pair $(\beta_1,\beta_2)$ satisfying the conditions of Remark~\ref{rem:VLIReduced}, one can group all of the corresponding linear inequalities into a single convex, possibly nonlinear, inequality valid for $\clconv(C_1\cup C_2)$. By Remark~\ref{rem:VLIReduced}, without any loss of generality, we focus on the case where $\beta_2=1$ and $\beta_1>0$ with $\beta_1c_{1,0}\geq c_{2,0}$ and $\beta_1c_1-c_2\notin\pm\intt\K_2^n$. Then by Lemma~\ref{lem:inconsistent}, $\beta_1c_1-c_2\notin-\K_2^n$. This leaves us two distinct cases to consider: $\beta_1c_1-c_2\in\bd\K_2^n$ and $\beta_1c_1-c_2\notin\pm\K_2^n$.

\begin{remark}\label{rem:linear}
Let $C_1, C_2$ satisfy the second-order cone disjunctive setup.
For any $\beta>0$ such that $\beta c_{1,0}\geq c_{2,0}$ and $\beta c_1-c_2\in\bd\K_2^n$, the inequality
\begin{equation}\label{eq:linear}
\beta c_1^\top x\geq c_{2,0}
\end{equation}
is valid for $\clconv(C_1\cup C_2)$ and dominates all valid linear inequalities that satisfy \eqref{eq:VLIReduced} with $\beta_1=\beta$ and $\beta_2=1$.
\end{remark}

\begin{proof}
The validity of \eqref{eq:linear} follows easily from $\beta c_{1,0}\geq c_{2,0}$ for $C_1$ and $\beta c_1-c_2\in\K_2^n$ for $C_2$. Let $\mu^\top x\geq c_{2,0}$ be a valid inequality  that satisfies \eqref{eq:VLIReduced} with $\beta_1=\beta$ and $\beta_2=1$. Then $\mu-\beta c_1=\alpha_1\in\K_2^n$, and since $\beta c_1^\top x\geq c_{2,0}$ is valid as well, we have that $\mu^\top x\geq c_{2,0}$ is dominated unless $\alpha_1=0$.
\end{proof}

\begin{theorem}\label{thm:main}
Let $C_1, C_2$ satisfy the second-order cone disjunctive setup.
For any $\beta>0$ such that $\beta c_{1,0}\geq c_{2,0}$ and $\beta c_1-c_2\notin\pm\K_2^n$, the inequality
\begin{equation}\label{eq:main}
2c_{2,0}-(\beta c_1+c_2)^\top x\leq\sqrt{\left((\beta c_1-c_2)^\top x\right)^2+\cN_1(\beta)\left(x_n^2-\norm{\tilde{x}}^2\right)}
\end{equation}
with
\begin{equation}\label{eq:N_beta}
\cN_1(\beta):=\norm{\beta\tilde{c_1}-\tilde{c_2}}_2^2-(\beta c_{1,n}-c_{2,n})^2
\end{equation}
is valid for $\clconv(C_1\cup C_2)$ and implies all valid linear inequalities that satisfy \eqref{eq:VLIReduced} with $\beta_1=\beta$ and $\beta_2=1$.
\end{theorem}

\begin{proof}
Consider the set of vectors $\mu\in\R^n$ satisfying \eqref{eq:VLIReduced} with $\beta_1=\beta$ and $\beta_2=1$:
\begin{equation*}
\capM(\beta,1)\!:=\left\{\mu\in\R^n\!:\,\exists\alpha_1,\alpha_2\in\bd\K_2^n~\text{ s.t. }~\mu=\alpha_1+\beta c_1=\alpha_2+c_2\right\}.
\end{equation*}
Because $\beta c_1-c_2\notin\pm\K_2^n$, Moreau's decomposition theorem implies that there exist $\mu^*,\alpha_1^*\neq 0,\alpha_2^*\neq 0$ such that $\alpha_1^*\perp\alpha_2^*$ and $(\mu^*,\alpha_1^*,\alpha_2^*,\beta,1)$ satisfies \eqref{eq:VLIReduced}. Hence, the set $\capM(\beta,1)$ is in fact nonempty. We can write
\begin{align*}
\capM(\beta,1)=&\left\{\mu\in\R^n:\|\tilde{\mu}-\beta\tilde{c}_1\|_2=\mu_n-\beta c_{1,n},\,
\|\tilde{\mu}-\tilde{c}_2\|_2=\mu_n-c_{2,n}\right\}\\
=&\left\{\mu\in\R^n:
\begin{array}{c}
\|\tilde{\mu}-\tilde{c}_2\|_2=\|\tilde{\mu}-\beta\tilde{c}_1\|_2+\beta c_{1,n}-c_{2,n},\\
\|\tilde{\mu}-\beta\tilde{c}_1\|_2=\mu_n-\beta c_{1,n}
\end{array}
\right\}.
\end{align*}
After taking the square of both sides of the first equation in $\capM(\beta,1)$, noting $\beta c_1-c_2\notin-\K_2^n$, and replacing the term $\norm{\tilde{\mu}-\beta\tilde{c}_1}_2$ with $\mu_n-\beta c_{1,n}$, we arrive at
\begin{equation*}
\capM(\beta,1)=\left\{\mu\in\R^n:\;
\begin{array}{c}
\tilde{\mu}^\top(\beta\tilde{c}_1-\tilde{c}_2)-\mu_n(\beta c_{1,n}-c_{2,n})=\frac{\cM}{2},\\
\norm{\tilde{\mu}-\beta\tilde{c}_1}_2=\mu_n-\beta c_{1,n}
\end{array}
\right\}
\end{equation*}
where $\cM:=\beta^2(\norm{\tilde{c}_1}_2^2-c_{1,n}^2)-(\norm{\tilde{c}_2}_2^2-c_{2,n}^2)$.

Note that $x\in\clconv(C_1\cup C_2)$ implies
\begin{align*}
&\Rightarrow x\in\K_2^n\text{ and }\mu^\top x\geq c_{2,0}\;~\forall\mu\in\capM(\beta,1).\\
&\Leftrightarrow x\in\K_2^n\text{ and }\inf_\mu\left\{\mu^\top x:\;\mu\in\capM(\beta,1)\right\}\geq c_{2,0}.
\end{align*}
Unfortunately, the optimization problem stated above is non-convex due to the second equality constraint in the description of $\capM(\beta,1)$.
We show below that the natural convex relaxation for this problem is tight. Indeed, consider the relaxation
\begin{equation*}
\inf_\mu\left\{\mu^\top x:\;
\begin{array}{c}
\tilde{\mu}^\top(\beta\tilde{c}_1-\tilde{c}_2)-\mu_n(\beta c_{1,n}-c_{2,n})=\frac{\cM}{2},\\
\norm{\tilde{\mu}-\beta\tilde{c}_1}_2\leq\mu_n-\beta c_{1,n}
\end{array}
\right\}
\end{equation*}
The feasible region of this relaxation is the intersection of a hyperplane with a closed, convex cone shifted by the vector $\beta c_1$. Any solution which is feasible to the relaxation but not the original problem can be expressed as a convex combination of solutions feasible to the original problem. Because we are optimizing a linear function, this shows that the relaxation is equivalent to the original problem. Thus, we have
\begin{multline*}
x\in\clconv(C_1\cup C_2)\Rightarrow\\
x\in\K_2^n\text{ and }
\inf_\mu\left\{\mu^\top x:\;
\begin{array}{c}
\tilde{\mu}^\top(\beta\tilde{c}_1-\tilde{c}_2)-\mu_n(\beta c_{1,n}-c_{2,n})=\frac{\cM}{2},\\
\norm{\tilde{\mu}-\beta\tilde{c}_1}_2\leq\mu_n-\beta c_{1,n}
\end{array}
\right\}
\end{multline*}
which is exactly the same as
\begin{multline}
x\in\clconv(C_1\cup C_2)\Rightarrow\\
x\in\K_2^n\text{ and }
\inf_\mu\left\{\mu^\top x:\;
\begin{array}{c}
\tilde{\mu}^\top(\beta\tilde{c}_1-\tilde{c}_2)-\mu_n(\beta c_{1,n}-c_{2,n})=\frac{\cM}{2},\\
\mu-\beta c_1\in\K_2^n
\end{array}
\right\}.\label{eq:convex}
\end{multline}
The minimization problem in the last line above is feasible since $\mu^*$, defined at the beginning of the proof, is a feasible solution. Indeed, it is strictly feasible since $\alpha_1^*+\alpha_2^*$ is a recession direction of the feasible region and belongs to $\intt\K_2^n$. Hence, its dual problem is solvable whenever it is feasible, strong duality applies, and we can replace the problem in the last line with its dual without any loss of generality.

Considering the definition of $\cN_1(\beta)=\norm{\beta\tilde{c_1}-\tilde{c_2}}_2^2-(\beta c_{1,n}-c_{2,n})^2$ and the assumption that $\beta c_1-c_2\notin\pm\K_2^n$, we get $\cN_1(\beta)>0$. Then
\begin{align*}
x&\in\clconv(C_1\cup C_2)\\
&\Rightarrow\!x\in\K_2^n\text{ and }
\max_{\rho,\tau}\!\left\{\beta c_1^\top\rho+\frac{\cM}{2}\tau:
\begin{array}{c}
\rho+\tau\left(\begin{array}{c}\beta\tilde{c}_1\!-\!\tilde{c}_2\\-\beta c_{1,n}+c_{2,n}\end{array}\right)=x,\\
\rho\in\K_2^n
\end{array}
\right\}\geq c_{2,0}.\\
&\Leftrightarrow\!x\in\K_2^n\text{ and }
\max_\tau\!\left\{\beta c_1^\top x-\frac{\cN_1(\beta)}{2}\tau:x+\tau\left(\begin{array}{c}\!-\!\beta\tilde{c}_1\!+\!\tilde{c}_2\\ \beta c_{1,n}-c_{2,n}\end{array}\right)\in\K_2^n\right\}\geq c_{2,0},\\
&\text{and since the optimum solution will be on the boundary of feasible region,}\\
&\Leftrightarrow\!x\in\K_2^n
\text{ and }
\min\{\tau_-,\tau_+\}\leq\frac{2(\beta c_1^\top x-c_{2,0})}{\cN_1(\beta)}\\
&\quad\text{ where }
\tau_{\pm}:=\frac{(\beta c_1-c_2)^\top x\pm\sqrt{((\beta c_1-c_2)^\top x)^2+\cN_1(\beta)(x_n^2-\norm{\tilde{x}}_2^2)}}{\cN_1(\beta)}.\\
&\Leftrightarrow\!x\in\K_2^n\text{ and }\tau_-\leq\frac{2(\beta c_1^\top x-c_{2,0})}{\cN_1(\beta)}.\\
&\Leftrightarrow\!x\in\K_2^n\text{ and }\cN_1(\beta)\tau_-\leq 2(\beta c_1^\top x-c_{2,0}).
\end{align*}
Rearranging the terms of the inequality in the last expression above yields \eqref{eq:main}.
\end{proof}

The next two observations follow directly from the proof of Theorem~\ref{thm:main}.

\begin{remark}\label{rem:convex}
Under the assumptions of Theorem~\ref{thm:main}, the set of points that satisfy \eqref{eq:main} in $\K_2^n$ is convex.
\end{remark}

\begin{proof}
The inequality~\eqref{eq:main} is equivalent to \eqref{eq:convex} by construction. The left-hand side of \eqref{eq:convex} is a concave function of $x$ written as the pointwise-infimum of linear functions, while the right-hand side is a constant.
\end{proof}

\begin{remark}
Inequality~\eqref{eq:main} reduces to the linear inequality \eqref{eq:linear} in $\K_2^n$ when $\beta c_1-c_2\in\bd\K_2^n$.
\end{remark}

\begin{proof}
When $\beta c_1-c_2\in\bd\K_2^n$, $\cN_1(\beta)=0$. Together with $x\in\K_2^n$, this also implies $(\beta c_1-c_2)^\top x\geq 0$, and hence, \eqref{eq:main} of Theorem~\ref{thm:main} becomes $2c_{2,0}-(\beta c_1+c_2)^\top x\leq(\beta c_1-c_2)^\top x$. This is equivalent to \eqref{eq:linear}.
\end{proof}

When $c_{1,0}>c_{2,0}$, by Proposition~\ref{prop:VLIReduced} and Remark~\ref{rem:VLIReduced}, the family of inequalities given in Remark~\ref{rem:linear} and Theorem~\ref{thm:main} is sufficient to describe $\clconv(C_1\cup C_2)$. On the other hand, when $c_{1,0}=c_{2,0}$, we also need to consider valid linear inequalities that satisfy \eqref{eq:VLIReduced} with $\beta_1=1$ and $\beta_2=\beta>0$ where $\beta c_{2,0}\geq c_{1,0}$ and $\beta c_2-c_1\notin\pm\intt\K_2^n$. Following Remark~\ref{rem:linear} and Theorem~\ref{thm:main} and letting $\cN_2(\beta):=\norm{\tilde{c_1}-\beta\tilde{c_2}}_2^2-(c_{1,n}-\beta c_{2,n})^2$, such linear inequalities can be summarized into the inequalities
\begin{align}
&\beta c_2^\top x\geq c_{2,0},\text{ and}\\
&2c_{2,0}-(c_1+\beta c_2)^\top x\leq\sqrt{\left((c_1-\beta c_2)^\top x\right)^2+\cN_2(\beta)\left(x_n^2-\norm{\tilde{x}}^2\right)}\label{eq:main2}
\end{align}
when $\beta c_2-c_1\in\bd\K_2^n$ and $\beta c_2-c_1\notin\pm\K_2^n$, respectively. In the remainder of this section, we continue to focus on the case where $\beta_2=1$ and $\beta_1=\beta>0$ with the understanding that our results are also applicable to the symmetric situation.

\subsection{A Conic Quadratic Form}\label{sec:sub:ConicQuadratic}

While having a convex valid inequality is nice in general, there are certain cases where \eqref{eq:main} can be expressed in conic quadratic form.

\begin{proposition}\label{prop:ConicForm}
Let $C_1, C_2$ satisfy the second-order cone disjunctive setup, and let $\beta>0$ be such that $\beta c_{1,0}\geq c_{2,0}$ and $\beta c_1-c_2\notin\pm\K_2^n$. Let $x\in\K_2^n$ be a point for which
\begin{equation}\label{eq:symmetry}
-2c_{2,0}+(\beta c_1+c_2)^\top x\leq
\sqrt{\left((\beta c_1-c_2)^\top x\right)^2+\cN_1(\beta)\left(x_n^2-\norm{\tilde{x}}^2\right)}
\end{equation}
holds with $\cN_1(\beta)$ defined as in \eqref{eq:N_beta}. Then $x$ satisfies \eqref{eq:main} if and only if it satisfies the conic quadratic inequality
\begin{equation}\label{eq:ConicForm}
\cN_1(\beta)x+2(c_2^\top x-c_{2,0})\left(\begin{array}{c}\beta\tilde{c}_1-\tilde{c}_2\\-\beta c_{1,n}+c_{2,n}\end{array}\right)\in\K_2^n.
\end{equation}
Furthermore, if \eqref{eq:symmetry} holds for all $x\in\clconv(C_1\cup C_2)$, then \eqref{eq:ConicForm} is valid for $\clconv(C_1\cup C_2)$ and implies \eqref{eq:main}.
\end{proposition}

\begin{proof}
Let $x\in\K_2^n$ be a point for which \eqref{eq:symmetry} holds. Then $x$ satisfies \eqref{eq:main} if and only if it satisfies
\begin{equation*}
|2c_{2,0}-(\beta c_1+c_2)^\top x|\leq\sqrt{\left((\beta c_1-c_2)^\top x\right)^2+\cN_1(\beta)\left(x_n^2-\norm{\tilde{x}}_2^2\right)}.
\end{equation*}
We can take the square of both sides without any loss of generality and rewrite this inequality as
\begin{align*}
&\left(2c_{2,0}-(\beta c_1+c_2)^\top x\right)^2\leq\left((\beta c_1-c_2)^\top x\right)^2+\cN_1(\beta)\left(x_n^2-\norm{\tilde{x}}_2^2\right)\\
&\quad\Leftrightarrow 4(\beta c_1^\top x-c_{2,0})(c_2^\top x-c_{2,0})\leq \cN_1(\beta)\left(x_n^2-\norm{\tilde{x}}_2^2\right).
\end{align*}
Because $\beta c_1-c_2\notin\pm\K_2^n$, we have $\cN_1(\beta)>0$, and the above inequality is equivalent to
\begin{equation*}
0\leq \cN_1(\beta)^2\left(x_n^2-\norm{\tilde{x}}_2^2\right)-4\cN_1(\beta)(\beta c_1^\top x-c_{2,0})(c_2^\top x-c_{2,0}).
\end{equation*}
The right-hand side of this inequality is identical to
\begin{equation*}
\left(\cN_1(\beta)x_n-2(c_2^\top x\!-\!c_{2,0})(\beta c_{1,n}\!-\!c_{2,n})\right)^2-\norm{\cN_1(\beta)\tilde{x}+2(c_2^\top x\!-\!c_{2,0})(\beta\tilde{c}_1\!-\!\tilde{c}_2)}_2^2.
\end{equation*}
Therefore, we arrive at
\begin{equation*}
\norm{\cN_1(\beta)\tilde{x}+2(c_2^\top x\!-\!c_{2,0})(\beta\tilde{c}_1\!-\!\tilde{c}_2)}_2^2\leq\left(\cN_1(\beta)x_n-2(c_2^\top x\!-\!c_{2,0})(\beta c_{1,n}\!-\!c_{2,n})\right)^2.
\end{equation*}
Let
\begin{align*}
&\cA(x):=\norm{\cN_1(\beta)\tilde{x}+2(c_2^\top x-c_{2,0})(\beta\tilde{c}_1-\tilde{c}_2)}_2\text{ and }\\
&\cB(x):=\cN_1(\beta)x_n-2(c_2^\top x-c_{2,0})(\beta c_{1,n}-c_{2,n}).
\end{align*}
We have just proved that $x$ satisfies \eqref{eq:main} if and only if it satisfies $\cA(x)^2\leq \cB(x)^2$. In order to finish the proof, all we need to show is that $\cA(u)^2\leq \cB(u)^2$ is equivalent to $\cA(u)\leq \cB(u)$ for all $u\in\K_2^n$. It will be enough to show that either $\cA(u)+\cB(u)>0$ or $\cA(u)=\cB(u)=0$ holds for all $u\in\K_2^n$. Suppose $\cA(u)+\cB(u)\leq 0$ for some $u\in\K_2^n$. Using the triangle inequality, we can write
\begin{align*}
0&\geq \cA(u)+\cB(u)\\
&=\norm{\cN_1(\beta)\tilde{u}+2(c_2^\top u-c_{2,0})(\beta\tilde{c}_1-\tilde{c}_2)}_2\\
&\qquad+\cN_1(\beta)u_n-2(c_2^\top u-c_{2,0})(\beta c_{1,n}-c_{2,n})\\
&\geq-\cN_1(\beta)\norm{\tilde{u}}_2+2|c_2^\top u-c_{2,0}|\norm{\beta\tilde{c}_1-\tilde{c}_2}_2\\
&\qquad+\cN_1(\beta)u_n-2|c_2^\top u-c_{2,0}||\beta c_{1,n}-c_{2,n}|\\
&=\cN_1(\beta)(u_n-\norm{\tilde{u}}_2)+2|c_2^\top u-c_{2,0}|(\norm{\beta\tilde{c}_1-\tilde{c}_2}_2-|\beta c_{1,n}-c_{2,n}|).
\end{align*}
Because $u\in\K_2^n$ and $\beta c_1-c_2\notin\pm\K_2^n$, we have $u_n-\norm{\tilde{u}}_2\geq 0$ and $\norm{\beta\tilde{c}_1-\tilde{c}_2}_2-|\beta c_{1,n}-c_{2,n}|>0$. Hence, $c_2^\top u=c_{2,0}$. This implies $\cA(u)+\cB(u)=\cN_1(\beta)(u_n+\norm{\tilde{u}}_2)$ which is strictly positive unless $u=0$, but then $\cA(u)=\cB(u)=0$.

The second claim of the proposition follows immediately from the first under the hypothesis that \eqref{eq:symmetry} holds for all $x\in\clconv(C_1\cup C_2)$.
\end{proof}

We next give a sufficient condition, based on a property of the intersection of $C_1$ and $C_2$, under which \eqref{eq:symmetry} is satisfied by every point in $\K_2^n$. Note that this condition thus allows our convex inequality \eqref{eq:main} to be represented in an equivalent conic quadratic form \eqref{eq:ConicForm}.

\begin{proposition}\label{prop:CQR}
Let $C_1, C_2$ satisfy the second-order cone disjunctive setup.
Let $\beta>0$ be such that $\beta c_{1,0}\geq c_{2,0}$ and $\beta c_1-c_2\notin\pm\K_2^n$. Then \eqref{eq:symmetry} holds for all $x\in\K_2^n$ that satisfy $\beta c_1^\top x\leq c_{2,0}$ or $c_2^\top x\leq c_{2,0}$.
Furthermore, if
\begin{equation}\label{eq:CQRCondition}
\{x\in\K_2^n:\;\beta c_1^\top x>c_{2,0},c_2^\top x>c_{2,0}\}=\emptyset,
\end{equation}
then \eqref{eq:symmetry} holds for all $x\in\K_2^n$ and \eqref{eq:ConicForm} is equivalent to \eqref{eq:main}.
\end{proposition}

\begin{proof}

Let $x\in\K_2^n$ satisfy $\beta c_1^\top x\leq c_{2,0}$ or $c_2^\top x\leq c_{2,0}$. Using Theorem~\ref{thm:main} on the disjunction $-\beta c_1^\top u\geq -c_{2,0}$ or $-c_2^\top u\geq -c_{2,0}$ shows that $x$ satisfies \eqref{eq:symmetry}. The second claim of the proposition now follows immediately from Proposition~\ref{prop:ConicForm} and \eqref{eq:CQRCondition}.
\end{proof}

Condition~\eqref{eq:CQRCondition} of Proposition~\ref{prop:CQR}, together with the results of Proposition~\ref{prop:ConicForm} and Theorem~\ref{thm:main}, identifies cases in which \eqref{eq:main} can be expressed in an equivalent conic quadratic form. In a split disjunction on the cone $\K_2^n$, it is easy to see using Lemma~\ref{lem:closure} that $C_1$ and $C_2$ are both nonempty and $\conv(C_1\cup C_2)\neq\K_2^n$ if and only if $c_1,c_2\notin\pm\K_2^n$ and $c_{1,0}=c_{2,0}=1$. For a proper two-sided split disjunction, $C_1\cap C_2=\emptyset$; hence, \eqref{eq:CQRCondition} is trivially satisfied with $\beta=1$.

\section{When does a Single Inequality Suffice?}\label{sec:SingleConvex}

In this section we give two conditions under which a single convex inequality of the type derived in Theorem~\ref{thm:main} describes $\clconv(C_1\cup C_2)$ completely, together with the cone constraint $x\in\K_2^n$. The main result of this section is Theorem~\ref{thm:SingleIneq} which we state below.

\begin{theorem}\label{thm:SingleIneq}
Let $C_1, C_2$ satisfy the second-order cone disjunctive setup with $c_1-c_2\notin\pm\K_2^n$. Then the inequality
\begin{equation}\label{eq:SingleIneq}
2c_{2,0}-(c_1+c_2)^\top x\leq\sqrt{\left((c_1-c_2)^\top x\right)^2+\cN\left(x_n^2-\norm{\tilde{x}}^2\right)}
\end{equation}
is valid for $\clconv(C_1\cup C_2)$ with $\cN:=\norm{\tilde{c_1}-\tilde{c_2}}_2^2-(c_{1,n}-c_{2,n})^2$. Furthermore,
\begin{equation*}
\clconv(C_1\cup C_2)=\{x\in\K_2^n:\;x\text{ satisfies }\eqref{eq:SingleIneq}\}
\end{equation*}
when, in addition,
\begin{enumerate}[(i)]
\item $c_1\in\K_2^n$, or $c_2\in\K_2^n$, or
\item $c_{1,0}=c_{2,0}\in\{\pm 1\}$ and undominated valid linear inequalities that are tight on both $C_1$ and $C_2$ are sufficient to describe $\clconv(C_1\cup C_2)$.
\end{enumerate}
\end{theorem}

The proof of Theorem~\ref{thm:SingleIneq} will require additional results on the structure of undominated valid linear inequalities. These are the subject of the next section.

\subsection{Further Properties of Undominated Valid Linear Inequalities}\label{sec:sub:VLIFurtherProperties}

In this section we consider the disjunction $c_1^\top x\geq c_{1,0}\vee c_2^\top x\geq c_{2,0}$ on a regular cone $\K$ and refine the results of Section~\ref{sec:sub:VLIProperties} on the structure of undominated valid linear inequalities. The results that we are going to present in this section hold for any regular cone $\K$.

The lemma below shows that the statement of Proposition~\ref{prop:VLIReduced} can be strengthened substantially when $c_1\in\K^*$ or $c_2\in\K^*$.

\begin{lemma}\label{lem:DualCone}
Let $C_1, C_2$ satisfy the basic disjunctive setup.
Suppose $c_1\in\K^*$ or $c_2\in\K^*$. Then, up to positive scaling, any undominated valid linear inequality for $\clconv(C_1\cup C_2)$ has the form $\mu^\top x\geq c_{2,0}$ where $\mu$ satisfies \eqref{eq:VLIReduced} with $\beta_1=\beta_2=1$.
\end{lemma}

\begin{proof}
First note that having $c_i\in\K^*$ implies $\rec C_i=\K$. Therefore, when $c_{2,0}\leq 0$, we can use Lemma~\ref{lem:closure} to conclude $\clconv(C_1\cup C_2)=\K$. In this case all valid inequalities for $\clconv(C_1\cup C_2)=\K$ are implied by the cone constraint $x\in\K$, and the claim holds trivially because there are no undominated valid inequalities. Thus, we only need to consider the situation in which $c_{1,0}=c_{2,0}=1$.

Assume without any loss of generality that $c_2\in\K^*$. Let $\nu^\top x\geq\nu_0$ be a valid inequality of the form given in Proposition~\eqref{prop:VLIReduced}. Then $\nu_0=\min\{c_{1,0},c_{2,0}\}=1$, and there exist $\alpha_1,\alpha_2,\beta_1,\beta_2$ such that $(\nu,\alpha_1,\alpha_2,\beta_1,\beta_2)$ satisfies \eqref{eq:VLIReduced}. In particular, $\nu=\alpha_1+\beta_1c_1=\alpha_2+\beta_2c_2\in\K^*$, and $\min\{\beta_1,\beta_2\}=1$. We are going to show that $\nu^\top x\geq 1$ is either dominated or has itself an equivalent representation \eqref{eq:VLIReduced} of the type claimed in the lemma. There are two cases that we need to consider: $\beta_1>\beta_2$ and $\beta_1<\beta_2$.

First suppose $\beta_1>\beta_2$. Then $\beta_2=1$ and $\alpha_1+\beta_1c_1=\alpha_2+c_2$. Having $\alpha_2=0$ contradicts Assumption~\ref{As:A1} through Lemma~\ref{lem:inconsistent}; therefore, $\alpha_2\neq 0$. Let $\epsilon'$ be such that $0<\epsilon'\leq\frac{\beta_1-1}{\beta_1}$, and define $\alpha_1':=(1-\epsilon')\alpha_1+\epsilon'c_2$, $\beta_1':=(1-\epsilon')\beta_1$, $\alpha_2':=(1-\epsilon')\alpha_2$ and $\mu:=\nu-\epsilon'\alpha_2$. The inequality $\mu^\top x\geq 1$ is valid for $\clconv(C_1\cup C_2)$ because $(\mu,1,\alpha_1',\alpha_2',\beta_1',1)$ satisfies \eqref{eq:VLI}. Furthermore, $\mu^\top x\geq 1$ dominates $\nu^\top x\geq 1$ since $\nu-\mu=\epsilon'\alpha_2\in\K^*\setminus\{0\}$.

Now suppose $\beta_2>\beta_1=1$. Observe that $(\nu,1,\alpha_1,\alpha_2+(\beta_2-1)c_2,1,1)$ is also a solution satisfying \eqref{eq:VLI}. If $\alpha_2+(\beta_2-1)c_2\in\intt\K^*$, we can find a valid inequality that dominates $\nu^\top x\geq 1$ as in the proof of Proposition~\ref{prop:VLIReduced}. Otherwise, $\alpha_2+(\beta_2-1)c_2\in\bd\K^*$ and $\nu^\top x\geq 1$ has the form claimed in the lemma since $(\nu,\alpha_1,\alpha_2+(\beta_2-1)c_2,1,1)$ satisfies \eqref{eq:VLIReduced}.
\end{proof}

When $c_{1,0}=c_{2,0}\in\{\pm 1\}$, a similar result holds for undominated valid linear inequalities that are tight on both $C_1$ and $C_2$.

\begin{lemma}\label{lem:tight}
Let $C_1, C_2$ satisfy the basic disjunctive setup with $c_{1,0}=c_{2,0}\in\{\pm 1\}$.
Then, up to positive scaling, any undominated valid linear inequality for $\clconv(C_1\cup C_2)$ that is tight on both $C_1$ and $C_2$ has the form $\mu^\top x\geq c_{2,0}$ where $\mu$ satisfies \eqref{eq:VLIReduced} with $\beta_1=\beta_2=1$.
\end{lemma}

\begin{proof}
Let $\mu^\top x\geq\mu_0$ be an undominated valid inequality for $\clconv(C_1\cup C_2)$ that is tight on both $C_1$ and $C_2$. Using Proposition~\ref{prop:VLIReduced}, we can assume that $\mu_0=c_{1,0}=c_{2,0}$ and there exist $\alpha_1,\alpha_2,\beta_1,\beta_2$ such that $(\mu,\alpha_1,\alpha_2,\beta_1,\beta_2)$ satisfies \eqref{eq:VLIReduced}. In particular, $\min\{\beta_1\mu_0,\beta_2\mu_0\}=\mu_0$.

Now consider the following pair of minimization problems
\begin{equation*}
\inf_x\{\mu^\top x:\;x\in C_1\}\quad\text{and}\quad\inf_x\{\mu^\top x:\;x\in C_2\},
\end{equation*}
and their duals
\begin{align*}
&\sup_{\delta,\gamma}\{\delta\mu_0:\,\mu=\gamma+\delta c_1,\gamma\in\K^*,\delta\geq 0\}\text{ and}\\
&\sup_{\delta,\gamma}\{\delta\mu_0:\,\mu=\gamma+\delta c_2,\gamma\in\K^*,\delta\geq 0\}.
\end{align*}
The pairs $(\alpha_1,\beta_1)$ and $(\alpha_2,\beta_2)$ are feasible solutions to the first and second dual problems, respectively. Because $\mu^\top x\geq\mu_0$ is tight on both $C_1$ and $C_2$, we must have $\beta_1\mu_0\leq\mu_0=\min\{\beta_1\mu_0,\beta_2\mu_0\}$ and $\beta_2\mu_0\leq\mu_0=\min\{\beta_1\mu_0,\beta_2\mu_0\}$ by duality. This implies $\beta_1\mu_0=\beta_2\mu_0=\mu_0$ and $\beta_1=\beta_2=1$.
\end{proof}

\begin{proof}[of Theorem~\ref{thm:SingleIneq}]
The validity of \eqref{eq:SingleIneq} follows from Theorem~\ref{thm:main} by setting $\beta=1$. Lemmas~\ref{lem:DualCone} and \ref{lem:tight} show that we can limit ourselves to valid linear inequalities that satisfy \eqref{eq:VLIReduced} with $\beta_1=\beta_2=1$ to get a complete description of the closed convex hull. When this is the case, the implication in \eqref{eq:convex} in the proof of Theorem~\ref{thm:main} is actually an equivalence.
\end{proof}

\subsection{A Topological Connection: Closedness of the Convex Hull}

Next, we identify an important case where the family of tight inequalities specified in Lemma~\ref{lem:tight} is rich enough to describe $\clconv(C_1\cup C_2)$ completely. The key ingredient is the closedness of $\conv(C_1\cup C_2)$.

\begin{proposition}\label{prop:closedness}
Consider $C_1, C_2$ defined as in \eqref{eq:DisjSet} with $c_{1,0},c_{2,0}\in\{0,\pm 1\}$. Suppose Assumptions~\ref{As:A1} and \ref{As:A2} hold.
Suppose $\conv(C_1\cup C_2)$ is closed. Then undominated valid linear inequalities that are strongly tight on both $C_1$ and $C_2$ are sufficient to describe $\conv(C_1\cup C_2)$, together with the cone constraint $x\in\K$.
\end{proposition}

\begin{proof}
Suppose $\conv(C_1\cup C_2)$ is closed. When $\conv(C_1\cup C_2)=\K$, no new inequalities are needed for a description of $\conv(C_1\cup C_2)$ and the claim holds trivially. Therefore, assume $\conv(C_1\cup C_2)\subsetneq\K$. We prove that given $u\in\K\setminus\conv(C_1\cup C_2)$, there exists an undominated valid inequality that separates $u$ from $\conv(C_1\cup C_2)$ and is strongly tight on both $C_1$ and $C_2$.

Let $v\in\intt\conv(C_1\cup C_2)\setminus(C_1\cup C_2)$. Note that such a point exists since otherwise, we have $\intt\conv(C_1\cup C_2)\subseteq C_1\cup C_2$ which implies $\conv(C_1\cup C_2)\subseteq C_1\cup C_2$ through the closedness of $C_1\cup C_2$. By Lemma~\ref{lem:closure}, this is possible only if $C_1\cup C_2=\K$ which we have already ruled out. Let $0<\lambda<1$ be such that $w:=(1-\lambda)u+\lambda v\in\bd\conv(C_1\cup C_2)$. Then $w\in\K\setminus(C_1\cup C_2)$ by the convexity of $\K\setminus(C_1\cup C_2)=\{x\in\K:\,c_1^\top x<c_{1,0},c_2^\top x<c_{2,0}\}$. Because $w\in\conv(C_1\cup C_2)$, there exist $x_1\in C_1$, $x_2\in C_2$, and $0<\kappa<1$ such that $w=\kappa x_1+(1-\kappa)x_2$. Furthermore, the fact that $w\in\bd\conv(C_1\cup C_2)$ implies that there exists an undominated valid inequality $\mu^\top x\geq\mu_0$ for $\conv(C_1\cup C_2)$ such that $\mu^\top w=\mu_0$. Because $\mu^\top w=\kappa\mu^\top x_1+(1-\kappa)\mu^\top x_2=\mu_0$, $\mu^\top x_1\geq\mu_0$, and $\mu^\top x_2\geq\mu_0$, it must be the case that $\mu^\top x_1=\mu^\top x_2=\mu_0$. Thus, the inequality $\mu^\top x\geq\mu_0$ is strongly tight on both $C_1$ and $C_2$. The only thing that remains is to show that $\mu^\top x\geq\mu_0$ separates $u$ from $\conv(C_1\cup C_2)$. To see this, observe that $u=\frac{1}{1-\lambda}(w-\lambda v)$ and that $\mu^\top v>\mu_0$ since $v\in\intt\conv(C_1\cup C_2)$. Hence, we conclude
\begin{equation*}
\mu^\top u=\frac{1}{1-\lambda}(\mu^\top w-\lambda\mu^\top v)<\mu_0.
\end{equation*}
\end{proof}

Proposition~\ref{prop:closedness} demonstrates the close relationship between the closedness of $\conv(C_1\cup C_2)$ and the sufficiency of valid linear inequalities that are tight on both $C_1$ and $C_2$. This motivates us to investigate the cases where $\conv(C_1\cup C_2)$ is closed.

The set $\conv(C_1\cup C_2)$ is always closed when $c_{1,0}=c_{2,0}=0$ (see, e.g., Rockafellar \cite[Corollary 9.1.3]{R1970}) or when $C_1$ and $C_2$ are defined by a split disjunction (see Dadush et al. \cite[Lemma 2.3]{DDV2011}). In Proposition~\ref{prop:ClosedSuff} below, we generalize the result of Dadush et al.: We give a sufficient condition for $\conv(C_1\cup C_2)$ to be closed and show that this condition is almost necessary. In Corollary~\ref{cor:ClosedSuff}, we show that the sufficient condition of Proposition~\ref{prop:ClosedSuff} can be rewritten in a more specialized form using conic duality when the base set is the regular cone $\K$. The proofs of these results are left to the appendix.

\begin{proposition}\label{prop:ClosedSuff}
Let $S\subset\R^n$ be a closed, convex, pointed set, $S_1:=\{x\in S:c_1^\top x\geq c_{1,0}\}$, and $S_2:=\{x\in S:c_2^\top x\geq c_{2,0}\}$ for $c_1,c_2\in\R^n$ and $c_{1,0},c_{2,0}\in\R$. Suppose $S_1\not\subseteq S_2$ and $S_1\not\supseteq S_2$. If
\begin{equation}
\begin{aligned}\label{eq:ClosedSuff}
&\{r\in\rec S:\;c_2^\top r=0\}\subseteq\{r\in\rec S:\;c_1^\top r\geq 0\}\text{ and }\\
&\{r\in\rec S:\;c_1^\top r=0\}\subseteq\{r\in\rec S:\;c_2^\top r\geq 0\},
\end{aligned}
\end{equation}
then $\conv(S_1\cup S_2)$ is closed. Conversely, if
\begin{enumerate}[(i)]
\item there exists $r^*\in\rec S$ such that $c_1^\top r^*<0=c_2^\top r^*$ and the problem $\inf_x\{c_2^\top x:\,x\in S_1\}$ is solvable, or
\item there exists $r^*\in\rec S$ such that $c_2^\top r^*<0=c_1^\top r^*$ and the problem $\inf_x\{c_1^\top x:\,x\in S_2\}$ is solvable,
\end{enumerate}
then $\conv(S_1\cup S_2)$ is not closed.
\end{proposition}

\begin{corollary}\label{cor:ClosedSuff}
Consider $C_1, C_2$ defined as in \eqref{eq:DisjSet} with $c_{1,0},c_{2,0}\in\{0,\pm 1\}$. Suppose Assumptions~\ref{As:A1} and \ref{As:A2} hold.
If there exist $\beta_1,\beta_2\in\R$ such that $c_1-\beta_2c_2\in\K^*$ and $c_2-\beta_1c_1\in\K^*$, then $\conv(C_1\cup C_2)$ is closed.
\end{corollary}

Let us define the following sets for ease of reference:
\begin{equation}
\begin{aligned}\label{eq:BetaSets}
&\capD_1:=\capD_1(c_1,c_2)=\{\beta_1\in\R:\;c_2-\beta_1c_1\in\K^*\},\\
&\capD_2:=\capD_2(c_1,c_2)=\{\beta_2\in\R:\;c_1-\beta_2c_2\in\K^*\}.
\end{aligned}
\end{equation}

Theorem~\ref{thm:SingleIneq}, Proposition~\ref{prop:closedness}, and Corollary~\ref{cor:ClosedSuff} imply that \eqref{eq:SingleIneq} is sufficient to describe $\conv(C_1\cup C_2)$ when $\capD_1$ and $\capD_2$ are both nonempty and $c_{1,0}=c_{2,0}\in\{\pm 1\}$. Nevertheless, it is easy to construct instances where $\capD_1$ or $\capD_2$ is empty. We explore these cases further in Section~\ref{sec:MultipleIneqs}.

Consider the case of $c_{1,0}=c_{2,0}\in\{0,\pm 1\}$. Then by Lemma~\ref{lem:inconsistent}, $c_1-c_2\notin\K_2^n$. Suppose also that
\begin{enumerate}[(a)]
\item condition (i) or (ii) of Theorem~\ref{thm:SingleIneq} is satisfied, and
\item $\{x\in\K_2^n:\,c_1^\top x>c_{1,0},c_2^\top x>c_{2,0}\}=\emptyset$.
\end{enumerate}
We note that statement~(a) holds, for instance, in the case of split disjunctions because $c_{1,0}=c_{2,0}=1$ and $\conv(C_1\cup C_2)$ is closed by Corollary~\ref{cor:ClosedSuff}. Moreover, statement~(b) simply means that the two sets $C_1$ and $C_2$ defined by the disjunction do not meet except, possibly, at their boundaries. This also holds for split disjunctions. Then by Theorem~\ref{thm:SingleIneq}, $\clconv(C_1\cup C_2)$ is completely described by \eqref{eq:SingleIneq} together with the cone constraint $x\in\K_2^n$. Furthermore, by Proposition~\ref{prop:CQR}, \eqref{eq:symmetry} is satisfied by every point in $\K_2^n$ with $\beta=1$ and by statement~(b), we have that \eqref{eq:SingleIneq} can be expressed in an equivalent conic quadratic form \eqref{eq:ConicForm}. Therefore, we conclude
\begin{equation*}
\clconv(C_1\cup C_2)=\left\{x\in\K_2^n:\;\cN x+2(c_2^\top x-c_{2,0})\left(\begin{array}{c}\tilde{c}_1-\tilde{c}_2\\- c_{1,n}+c_{2,n}\end{array}\right)\in\K_2^n\right\}
\end{equation*}
where $\cN:=\norm{\tilde{c_1}-\tilde{c_2}}_2^2-(c_{1,n}-c_{2,n})^2$. Thus, Theorem~\ref{thm:SingleIneq} and Proposition~\ref{prop:ConicForm}, together with Proposition~\ref{prop:CQR}, cover the results of \cite{MKV} and \cite{AJ2013} on split disjunctions on the cone $\K_2^n$ and significantly extend these results to more general two-term disjunctions.

\subsection{Example where a Single Inequality Suffices}\label{sec:sub:TightExample}

Consider the cone $\K_2^3$ and the disjunction $x_3\geq 1\,\vee\,x_1+x_3\geq 1$. Note that $c_1=e^3\in\K_2^3$ in this example. Hence, we can use Theorem~\ref{thm:SingleIneq} to characterize the closed convex hull:
\begin{equation*}
\clconv(C_1\cup C_2)=\left\{x\in\K_2^3:\;2-(x_1+2x_3)\leq\sqrt{x_3^2-x_2^2}\right\}.
\end{equation*}
Figures~\ref{fig:T}(a) and (b) depict the disjunctive set $C_1\cup C_2$ and the associated closed convex hull, respectively. In order to give a better sense of the convexification operation, we plot the points added to $C_1\cup C_2$ to generate the closed convex hull in Figure~\ref{fig:T}(c). We note that in this example the condition on the disjointness of the interiors of $C_1$ and $C_2$ that was required in Proposition~\ref{prop:CQR} is violated. Nevertheless, the inequality that we provide is still intrinsically related to the conic quadratic inequality \eqref{eq:ConicForm} of Proposition~\ref{prop:ConicForm}: The sets described by the two inequalities coincide in the region $\clconv(C_1\cup C_2)\setminus(C_1\cap C_2)$ as a consequence of Proposition~\ref{prop:CQR}. We display the corresponding cone for this example in Figure~\ref{fig:T}(d). Moreover, the resulting conic quadratic inequality is in fact not valid for some points in $\conv(C_1\cup C_2)$, which can be seen by contrasting Figures~\ref{fig:T}(c) and \ref{fig:T}(d).
\begin{figure}[h]
 \begin{center}
 \subfigure[$C_1\cup C_2$]{
 \includegraphics[scale=0.325]{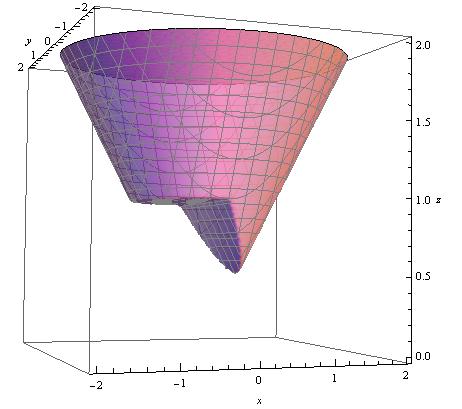}
 }
 \quad
 \subfigure[$\clconv(C_1\cup C_2)$]{
 \includegraphics[scale=0.32]{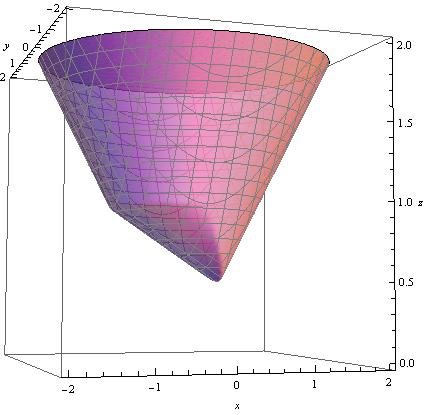}
 }
 \\
 \subfigure[$\clconv(C_1\cup C_2)\setminus(C_1\cup C_2)$]{
 \includegraphics[scale=0.305]{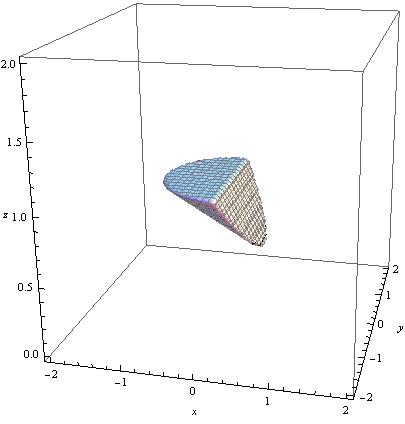}
 }
 \quad
 \subfigure[Underlying cone generating the convex inequality]{
 \includegraphics[scale=0.355]{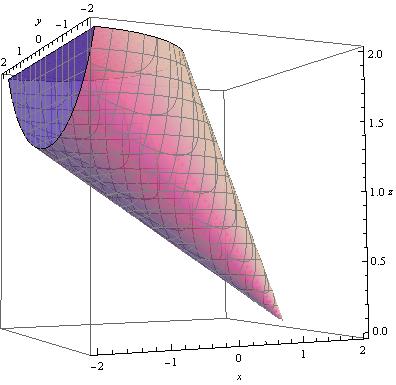}
 }
 \end{center}
 \caption{Sets associated with the disjunction $x_3\geq 1\,\vee\,x_1+x_3\geq 1$ on $\K_2^3$.}
 \label{fig:T}
\end{figure}

\section{When are Multiple Convex Inequalities Needed?}\label{sec:MultipleIneqs}

Lemma~\ref{lem:tight} allows us to simplify the characterization \eqref{eq:VLIReduced} of undominated valid linear inequalities which are tight on both $C_1$ and $C_2$ in the case $c_{1,0}=c_{2,0}\in\{\pm 1\}$. The next proposition shows the necessity of this assumption on $c_{1,0}$ and $c_{2,0}$.
Unfortunately, when $c_{1,0}\neq c_{2,0}$, undominated valid linear inequalities are tight on exactly one of the two sets $C_1$ and $C_2$. The proof of this result is left to the appendix.

\begin{proposition}\label{prop:nontight}
Let $C_1, C_2$ satisfy the basic disjunctive setup. If $c_{1,0}>c_{2,0}$, then every undominated valid linear inequality for $\clconv(C_1\cup C_2)$ is tight on $C_2$ but not on $C_1$.
\end{proposition}

This result, when combined with Proposition~\ref{prop:closedness}, yields the following corollary.

\begin{corollary}
Let $C_1, C_2$ satisfy the basic disjunctive setup with $c_{1,0}>c_{2,0}$. If $\clconv(C_1\cup C_2)\neq\K$, then $\conv(C_1\cup C_2)$ is not closed.
\end{corollary}

\begin{proof}
Suppose $\conv(C_1\cup C_2)$ is closed, and let $x\in\K\setminus\conv(C_1\cup C_2)$. By Proposition~\ref{prop:closedness}, there exists an undominated valid linear inequality which cuts off $x$ from $\conv(C_1\cup C_2)$ and is tight on both $C_1$ and $C_2$. This contradicts Proposition~\ref{prop:nontight}.
\end{proof}

\subsection{Describing the Closed Convex Hull}\label{sec:sub:ConvexHull}

As Proposition~\ref{prop:nontight} hints, there are cases where valid linear inequalities that have $\beta_1=\beta_2=1$ in \eqref{eq:VLIReduced} may not be sufficient to describe $\clconv(C_1\cup C_2)$. In this section, we study these cases when $\K=\K_2^n$ and outline a procedure to find closed-form expressions describing $\clconv(C_1\cup C_2)$. Note that, by Theorem~\ref{thm:SingleIneq}, when $c_1\in\K_2^n$ or $c_2\in\K_2^n$, the convex valid inequality \eqref{eq:SingleIneq} is sufficient to describe $\clconv(C_1\cup C_2)$. Similarly, when the sets $\capD_1$ and $\capD_2$ defined in \eqref{eq:BetaSets} are both nonempty, \eqref{eq:SingleIneq} is sufficient to describe $\clconv(C_1\cup C_2)$. Hence, these cases are not of interest to us in this section, and we assume that $c_1,c_2\notin\K_2^n$ and at least one of $\capD_1$ and $\capD_2$ is empty. We analyze the remaining cases through a breakdown based on whether $\capD_1$ and $\capD_2$ are empty or not. Note that for now we do not make the assumption that $c_{1,0}\geq c_{2,0}$; therefore, the roles of $C_1$ and $C_2$ are completely symmetric.

Let $C_1$ and $C_2$ be defined as in \eqref{eq:DisjSet} with $\K=\K_2^n$, and let
\begin{align*}
\capB_1:=\capB_1(c_1,c_2)=\{\beta_1\in\R_+:\;\beta_1 c_{1,0}\geq c_{2,0},\;c_2-\beta_1 c_1\not\in\pm\intt\K_2^n\},\\
\capB_2:=\capB_2(c_1,c_2)=\{\beta_2\in\R_+:\;\beta_2 c_{2,0}\geq c_{1,0},\;c_1-\beta_2 c_2\not\in\pm\intt\K_2^n\}.
\end{align*}
Using Remark~\ref{rem:VLIReduced}, it is clear that one only needs to consider $\beta_1\in\capB_1$ in order to capture all undominated valid linear inequalities that have a representation with $\beta_2=1$ in \eqref{eq:VLIReduced}. Similarly, one only needs to consider $\beta_2\in\capB_2$ to capture the inequalities that have a representation with $\beta_1=1$ in \eqref{eq:VLIReduced}. Therefore, by Theorem~\ref{thm:main},
\begin{equation*}
\clconv(C_1\cup C_2)=
\{x\in\K_2^n:x\text{ satisfies }\eqref{eq:main}\;\forall\beta\in\capB_1\text{ and }\eqref{eq:main2}\;\forall\beta\in\capB_2\}.
\end{equation*}
Note that for any $\beta_1\in\capB_1$ and $\beta_2\in\capB_2$, we have $\cN_1(\beta_1)\geq 0$ and $\cN_2(\beta_2)\geq 0$; hence, the right hand sides of the inequalities above are well defined for any $x\in\K_2^n$. Using the structure of $\K_2^n$, we can process the definition of $\capB_1$ above and arrive at
\begin{align}
\!\!\!\!\!\capB_1\!\!&=\!\!\{\beta\in\R_+:\;\beta c_{1,0}\geq c_{2,0},\;\norm{\tilde{c}_2-\beta\tilde{c_1}}_2\geq |c_{2,n}-\beta c_{1,n}|\}\notag\\
&=\!\left\{\beta\in\R_+\!:\begin{array}{l}\beta c_{1,0}\geq c_{2,0},\\
\!\left(\|\tilde{c}_1\|_2^2\!-\!c_{1,n}^2\right)\beta^2 \!-\! 2\left(\tilde{c}_1^\top\tilde{c}_2\!-\!c_{1,n}c_{2,n}\right)\beta\!+\! \|\tilde{c}_2\|_2^2\!-\!c_{2,n}^2\!\geq 0\end{array}\right\}.\label{eq:B1}
\end{align}
Similarly,
\begin{equation*}
\capB_2\!\!=
\!\!\left\{\beta\in\R_+\!:\begin{array}{l}\beta c_{2,0}\geq c_{1,0},\\
\!\left(\|\tilde{c}_2\|_2^2\!-\!c_{2,n}^2\right)\beta^2 \!-\! 2\left(\tilde{c}_1^\top\tilde{c}_2\!-\!c_{1,n}c_{2,n}\right)\beta\!+\! \|\tilde{c}_1\|_2^2\!-\!c_{1,n}^2\!\geq 0\end{array}\right\}.
\end{equation*}
In the following, based on the feasibility status of $\capD_1,\capD_2$, we show that the description of $\capB_1$ and $\capB_2$ can be simplified.

\subsubsection{$\capD_1=\emptyset$ and $\capD_2\neq\emptyset$}

First note that having $\capD_1=\emptyset$ implies $c_1\notin-\intt\K_2^n$. Furthermore, any $\beta\in\capD_2$ must satisfy $\beta>0$ since otherwise, either $c_1\in\K_2^n$ or $\capD_1\neq\emptyset$. Then by Lemma~\ref{lem:inconsistent}, $\beta c_{2,0}<c_{1,0}$, and $c_{2,0}\leq c_{1,0}$ because $\beta>0$ and $c_{1,0},c_{2,0}\in\{0,\pm 1\}$. This also implies that we cannot have $\capD_1=\emptyset$ and $\capD_2\neq\emptyset$ when $c_{1,0}=c_{2,0}=0$.

Recall that when $c_{1,0}>c_{2,0}$, Remark~\ref{rem:VLIReduced} showed that we can assume $\beta_2=1$ in an undominated valid linear inequality that satisfies \eqref{eq:VLIReduced}. In Proposition~\ref{prop:TightOnOne} below, we prove a similar result for the case $c_{1,0}=c_{2,0}\in\{\pm 1\}$ when $\capD_1=\emptyset$ and $\capD_2\neq\emptyset$. Its proof is left to the appendix.

\begin{proposition}\label{prop:TightOnOne}
Let $C_1, C_2$ satisfy the basic disjunctive setup with $c_{1,0}=c_{2,0}\in\{\pm 1\}$.
If $\capD_1=\emptyset$ and $\capD_2\neq\emptyset$, then every undominated valid linear inequality for $\clconv(C_1\cup C_2)$ has the form $\mu^\top x\geq c_{2,0}$ where $\mu$ satisfies \eqref{eq:VLIReduced} with $\beta_2=1$.
\end{proposition}

As a result of Proposition~\ref{prop:TightOnOne}, we have
\begin{equation*}
\clconv(C_1\cup C_2)=
\{x\in\K_2^n:x\text{ satisfies }\eqref{eq:main}\;\forall\beta\in\capB_1\}.
\end{equation*}

When $c_1\in-\bd\K_2^n$, we have $\|\tilde{c}_1\|_2^2-c_{1,n}^2=0$ and the second constraint in the definition of $\capB_1$ reduces to a linear inequality. Thus, $\capB_1$ is a closed interval of the nonnegative half-line in this case. On the other hand, when $c_1\notin\pm\K_2^n$, the structure of $\capB_1$ can be slightly more complicated. Since $c_1\notin\pm\K_2^n$, we can write $\capB_1$ as
\begin{equation*}
\capB_1\!=
\!\{\beta\in\R_+:\;\beta c_{1,0}\geq c_{2,0}\}\!\bigcap\left(\{\beta\!:\beta\leq\beta_1^-\}\cup\{\beta\!:\beta\geq\beta_1^+\}\right).
\end{equation*}
where $\beta_1^\pm$ are the values of $\beta$ that satisfy the second constraint in \eqref{eq:B1} with equality:
\begin{equation*}
\beta_1^\pm=\frac{\tilde{c}_1^\top\tilde{c}_2-c_{1,n}c_{2,n}\pm\sqrt{\left(\tilde{c}_1^\top\tilde{c}_2-c_{1,n}c_{2,n}\right)^2- \left(\|\tilde{c}_1\|_2^2-c_{1,n}^2\right)\left(\|\tilde{c}_2\|_2^2-c_{2,n}^2\right)}}{\|\tilde{c}_1\|_2^2-c_{1,n}^2}.
\end{equation*}
Recall that any $\beta\in\capD_2$ satisfies $\beta>0$ because $c_1\notin\K_2^n$ and $\capD_1=\emptyset$. Hence, $c_2-\frac{1}{\beta}c_1\in-\K^*$, and because $c_1\notin\pm\K^*$, the roots $\beta_1^\pm$ are well-defined: They are exactly the values of $\beta$ that yield $c_2-\beta c_1\in-\bd\K^*$. Suppose $\beta_1^-,\beta_1^+\in\{\beta\in\R_+:\beta c_{1,0}\geq c_{2,0}\}$. When $c_{2,0}\geq 0$, the inequality $\beta_1^-c_1^\top x\geq c_{2,0}$ dominates (or is equivalent to) all valid linear inequalities that correspond to $\beta\geq\beta^+$. Similarly, when $c_{2,0}=-1$, the inequality $\beta_1^+c_1^\top x\geq c_{2,0}$ dominates (or is equivalent to) all valid linear inequalities that correspond to $\beta\leq\beta^-$. Therefore, $\capB_1$ can always be reduced to a single closed interval of the nonnegative half-line in this case as well.

\subsubsection{$\capD_1=\capD_2=\emptyset$}

The hypothesis $\capD_1=\capD_2=\emptyset$ implies $c_1,c_2\notin-\intt\K_2^n$. Observe that for any $\beta\geq 0$, we must have $c_2-\beta_1c_1\notin\pm\intt\K_2^n$ since having $c_2-\beta c_1\in\intt\K_2^n$ contradicts $\capD_1=\emptyset$ and having $c_2-\beta c_1\in-\intt\K_2^n$ contradicts $\capD_2=\emptyset$. Similarly, for any $\beta\geq 0$, we must have $c_1-\beta c_2\notin\pm\intt\K_2^n$. Therefore, we can drop the second constraint in the definitions of $\capB_1$ and $\capB_2$ and write
\begin{align*}
&\capB_1=\{\beta\in\R_+:\;\beta c_{1,0}\geq c_{2,0}\},
&\capB_2=\{\beta\in\R_+:\;\beta c_{2,0}\geq c_{1,0}\}.
\end{align*}

\subsubsection{Finding the Best $\beta$}

Suppose $C_1$ and $C_2$ satisfy the second-order disjunctive setup. For ease of notation, let us define
\begin{align*}
\cR:=\cR(c_1,c_2,x) &=(c_1^\top x)^2 + (\|\tilde{c}_1\|_2^2-c_{1,n}^2)(x_n^2-\|\tilde{x}\|_2^2),\\
\cP:=\cP(c_1,c_2,x) &=(c_1^\top x)(c_2^\top x) + (\tilde{c}_1^\top\tilde{c}_2-c_{1,n}c_{2,n})(x_n^2-\|\tilde{x}\|_2^2),\\
\cQ:=\cQ(c_1,c_2,x) &=(c_2^\top x)^2 + (\|\tilde{c}_2\|_2^2-c_{2,n}^2)(x_n^2-\|\tilde{x}\|_2^2),
\end{align*}
and $f_1^{c_1,c_2,x}(\beta):=\beta c_1^\top x + \sqrt{\cR\beta^2-2\cP\beta+\cQ}$. Then
\begin{equation*}
\cR\beta^2-2\cP\beta+\cQ=\left((\beta c_1-c_2)^\top x\right)^2\!+\cN_1(\beta)\left(x_n^2-\|\tilde{x}\|_2^2\right).
\end{equation*}
Similarly, define $f_2^{c_1,c_2,x}(\beta):=\beta c_2^\top x + \sqrt{\cQ\beta^2-2\cP\beta+\cR}$ and note
\begin{equation*}
\cQ\beta^2-2\cP\beta+\cR=\left((c_1-\beta c_2)^\top x\right)^2\!+\cN_2(\beta)\left(x_n^2-\|\tilde{x}\|_2^2\right).
\end{equation*}
Through these definitions, we reach
\begin{align*}
\clconv(C_1\cup C_2)&=\left\{x\in\K_2^n:
\begin{array}{c}
2c_{2,0}-c_2^\top x\leq f_1^{c_1,c_2,x}(\beta_1)\;~\forall\beta_1\in B_1,\\
2c_{2,0}-c_1^\top x\leq f_2^{c_1,c_2,x}(\beta_2)\;~\forall\beta_2\in B_2
\end{array}
\right\}\\
&=\left\{x\in\K_2^n:
\begin{array}{c}
2c_{2,0}-c_2^\top x\leq\inf_{\beta_1\in B_1}f_1^{c_1,c_2,x}(\beta_1),\\
2c_{2,0}-c_1^\top x\leq\inf_{\beta_2\in B_2}f_2^{c_1,c_2,x}(\beta_2)
\end{array}
\right\}.
\end{align*}

For any $x\in\bd\K$, we have
\begin{align*}
f_1^{c_1,c_2,x}(\beta_1)=\max\{(2\beta_1c_1-c_2)^\top x,c_2^\top x\},
f_2^{c_1,c_2,x}(\beta_2)=\max\{(2\beta_2c_2-c_1)^\top x,c_1^\top x\}
\end{align*}
which can be easily minimized over $\beta_1\in B_1$ and $\beta_2\in B_2$. For any $x\in\intt K$, the expressions inside the square roots in the definitions of $f_1^{c_1,c_2,x}$ and $f_2^{c_1,c_2,x}$ are strictly positive for all $\beta_1\in B_1$ and $\beta_2\in B_2$, respectively. A function of the form $g(\beta)=a\beta+\sqrt{r\beta^2-2p\beta+q}$ is differentiable in $\beta$ where $r\beta^2-2p\beta+q>0$. Furthermore, it is concave if $p^2\geq qr$ and convex if $p^2\leq qr$. Whenever the infimum of $g$ over a closed interval of the real line is finite, it will be achieved either at a critical point of $g(\beta)$, where its derivative with respect to $\beta$ vanishes, or at one of the boundary points of $\capB(c_1,c_2)$. When $r>a^2$ and $g$ is convex, the critical point of $g$ is given by
\begin{equation*}
\beta^* = \frac{p}{r}-\frac{a}{r}\sqrt{\frac{p^2-qr}{a^2-r}},
\end{equation*}
and the corresponding value of the function at this critical point is
\begin{align*}
g(\beta^*)&=\frac{ap}{r}+\left(1-\frac{a^2}{r}\right)\sqrt{\frac{p^2-qr}{a^2-r}}\\
&=\frac{1}{r}\left(ap+\sign(r-a^2)\sqrt{|(p^2-qr)(a^2-r)|}\right).
\end{align*}
Replacing $p,q,r,a$ with $\cP,\cQ,\cR,c_1^\top x$, we define
\begin{equation*}
\beta_1^*:=\beta_1^*(c_1,c_2,x)\!=\!\frac{\cP}{\cR} - \frac{c_1^\top\!x}{\cR} \sqrt{\frac{\cP^2\!-\!\cQ\cR}{\left(c_{1,n}^2\!-\!\|\tilde{c}_1\|_2^2\right)\left(x_n^2\!-\!\|\tilde{x}\|_2^2\right)}}.
\end{equation*}
Note that $c_{1,n}^2-\norm{\tilde{c}_1}_2^2>0$ when $c_1\notin\pm\K^*$. Under these circumstances, for all $x\in\intt\K_2^n$ such that $\beta_1^*(c_1,c_2,x)\in\capB_1(c_1,c_2)$, we can enforce the inequality $2c_{2,0}-c_2^\top x\leq f_1^{c_1,c_2,x}(\beta_1^*)$. This inequality is only valid for those $x\in\K_2^n$ satisfying $\beta_1^*(c_1,c_2,x)\in\capB_1(c_1,c_2)$, but for these points, it completely summarizes all other valid inequalities of the form \eqref{eq:main} arising from $\beta\in\capB_1(c_1,c_2)$.

Similarly, we define
\begin{equation*}
\beta_2^*:=\beta_2^*(c_1,c_2,x)\!=\!\frac{\cP}{\cQ} - \frac{c_2^\top\!x}{\cQ} \sqrt{\frac{\cP^2\!-\!\cQ\cR}{\left(c_{2,n}^2\!-\!\|\tilde{c}_2\|_2^2\right)\left(x_n^2\!-\!\|\tilde{x}\|_2^2\right)}}.
\end{equation*}
Therefore, for all $x\in\intt\K_2^n$ such that $\beta_2^*(c_1,c_2,x)\in\capB_2(c_1,c_2)$, we have $2c_{2,0}-c_1^\top x\leq f_2^{c_1,c_2,x}(\beta_2^*)$ as a valid inequality completely summarizing all other valid inequalities of the form \eqref{eq:main2} arising from $\beta\in\capB_2(c_1,c_2)$.

\subsection{Example where Multiple Inequalities are Needed}\label{sec:sub:NonTightExample}

Consider $\K_2^3$ and the disjunction given by $-x_2\geq 0$ or $-x_3\geq -1$. Since $c_{1,0}>c_{2,0}$, by Proposition~\ref{prop:nontight}, we know that every undominated valid linear inequality for $\clconv(C_1\cup C_2)$ will be tight on $C_2$ but not on $C_1$. Therefore, we follow the approach outlined in Section~\ref{sec:sub:ConvexHull}. By noting that $c_1=-e^2\not\in\pm K_2^3$, we obtain $\capB_1(c_1,c_2)=\{\beta\in\R:\,\beta\geq 1\}$.
It is also clear that $\capB_2(c_1,c_2)=\emptyset$ in this case.

In this setup we have
\begin{align*}
\cR(c_1,c_2,x) &=
x_3^2-x_1^2,\\
\cP(c_1,c_2,x) &=
x_2x_3,\\
\cQ(c_1,c_2,x) &=
x_1^2+x_2^2,
\end{align*}
and the resulting $f_1^{c_1,c_2,x}(\beta)$ is a convex function of $\beta$. Hence we get 
\begin{align*}
\beta_1^*(c_1,c_2,x)
&=\frac{x_2x_3}{x_3^2-x_1^2}+\frac{x_2}{x_3^2-x_1^2}\sqrt{\frac{ x_2^2x_3^2-(x_1^2+x_2^2)(x_3^2-x_1^2)}{(-1)(x_3^2-x_1^2-x_2^2)}}\\
&=\frac{x_2x_3+|x_1|x_2}{x_3^2-x_1^2}=\frac{x_2}{x_3-|x_1|}
\end{align*}
where in the last equation we used the fact that $x\in\K_2^3$ and hence $x_3\geq 0$. This leads to $f_1^{c_1,c_2,x}(\beta_1^*)= |x_1|-\frac{x_2^2(x_3+|x_1|)}{x_3^2-x_1^2}= |x_1|-\frac{x_2^2}{x_3-|x_1|}$.

Therefore, for all $x\in\K_2^3$ such that $\beta_1^*\geq 1$, that is, $|x_1|\geq x_3-x_2$, we can enforce $2c_{2,0}-c_2^\top x\leq f_1^{c_1,c_2,x}(\beta_1^*)$ which translates to $-2+x_3\leq|x_1|-\frac{x_2^2}{x_3-|x_1|}$ in this example. Moreover, $\bd\capB_1(c_1,c_2)=\{1\}$, and for this particular value of $\beta=1$, using \eqref{eq:linear} in Remark~\ref{rem:linear}, we obtain $x_2\leq 1$ as a valid linear inequality for all $x\in\clconv(C_1\cup C_2)$. Putting these two inequalities together, we arrive at
\begin{equation*}
\clconv(C_1\cup C_2)=\left\{x\in\K_2^3\!:\;x_2\leq 1,\;1+|x_1|-x_3\leq\sqrt{1-\max\{0,x_2\}^2}\right\},
\end{equation*}
where both inequalities are convex (even when we ignore the constraint $x\in\K_2^3$). In fact, both inequalities describing $\clconv(C_1\cup C_2)$ are conic quadratic representable in a lifted space as expected.

In Figures~\ref{fig:NT}(a) and (b), we plot the disjunctive set and the resulting closed convex hull, respectively. In order to give a better picture of the convexification of the set, we show the points added due to the convex hull operation in Figure~\ref{fig:NT}(c).

\begin{figure}[h]
\begin{center}
\subfigure[$C_1\cup C_2$]{
\includegraphics[scale=0.33]{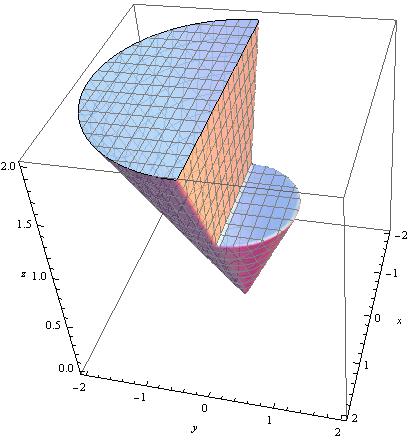}
}
\qquad\quad
\subfigure[$\clconv(C_1\cup C_2)$]{
\includegraphics[scale=0.33]{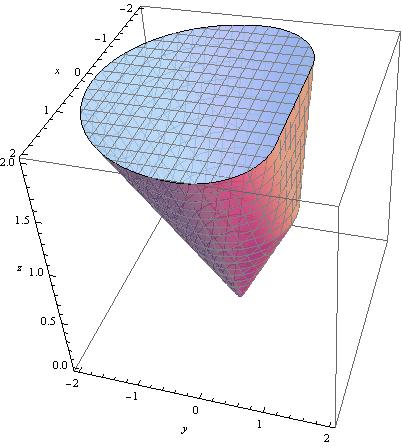}
}\\
\subfigure[$\clconv(C_1\cup C_2)\setminus(C_1\cup C_2)$]{
\includegraphics[scale=0.33]{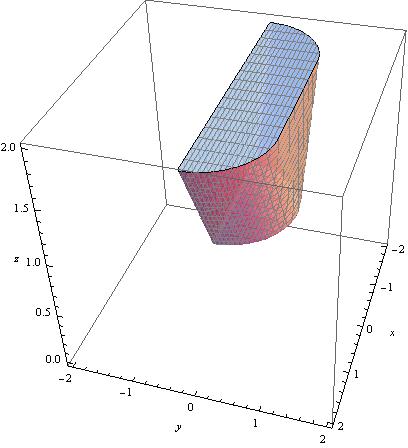}
}
\end{center}
\caption{Sets associated with the disjunction $-x_2\geq 0\,\vee\,-x_3\geq -1$ on $\K_2^3$.}
\label{fig:NT}
\end{figure}

\section{Elementary Split Disjunctions for $p$-order Cones}\label{sec:p-order}

Given $w\in\R^d$ and $p\in(1,\infty)$, recall that the $p$-norm $\norm{\cdot}_p:\R^d\rightarrow\R$ is defined as
\begin{equation*}
\norm{w}_p:=\left(\sum_{j=1}^d |w_j|^p\right)^{1/p}.
\end{equation*}
Its dual norm is the function $\norm{\xi}_p^*:=\max\{\xi^\top w:\,\norm{w}_p\leq 1\,\forall w\in\R^d\}$ and corresponds to the $q$-norm on $\R^d$ where $\frac{1}{p}+\frac{1}{q}=1$. In this section we consider the $n$-dimensional $p$-order cone $\K_p^n:=\left\{x\in\R^n:\norm{\tilde{x}}_p\leq x_n\right\}$, which is a regular cone and whose dual cone is simply $\K_q^n$.

The main result of this section shows that the techniques of the previous sections can be used to describe the convex hull of the set obtained by applying an elementary split disjunction on $\K_p^n$. Let $c_1=t_1e^i$, $c_2=-t_2e^i$, $c_{1,0},c_{2,0}\in\{0,\pm 1\}$ where $t_1,t_2>0$, $e^i$ is the $i^{\text{th}}$ standard unit vector, and $i\in\{1,\ldots,n-1\}$. Note that $\conv(C_1\cup C_2)$ is closed by Corollary~\ref{cor:ClosedSuff} and $\conv(C_1\cup C_2)=\K_p^n$ unless $c_{1,0}=c_{2,0}=1$ by Lemma~\ref{lem:closure}. Hence, we consider the case $c_{1,0}=c_{2,0}=1$ only. This result recovers and provides an independent proof of Corollary 1 in \cite{MKV}. Its proof follows the same outline as the proof of Theorem~\ref{thm:main} and is therefore left to the appendix.

\begin{corollary}\label{cor:p-order}
Let $C_1:=\{x\in\K_p^n:\,t_1x_i\geq 1\}$ and $C_2:=\{x\in\K_p^n:\;-t_2x_i\geq 1\}$ where $t_1,t_2>0$, $p\in(1,\infty)$, and $i\in\{1,\ldots,n-1\}$. Let $e^i$ denote the $i^{\text{th}}$ standard unit vector. Then
\begin{equation*}\label{eq:p-order}
\conv(C_1\cup C_2)=\left\{x\in\K_p^n:\;\norm{(t_1+t_2)\tilde{x}-2(t_2x_i+1)\tilde{e}^i}_p\leq(t_1+t_2)x_n\right\}.
\end{equation*}
\end{corollary}

\bibliographystyle{plain}
\bibliography{ref}

\newpage

\input{app_DisjunctionsOnLorentzCone}

\end{document}

%% file: app_DisjunctionsOnLorentzCone.tex
\section{Appendix}
\subsection{Proofs of Section \ref{sec:preliminaries}}

\begin{proof}[of Lemma~\ref{lem:closure}]
To prove the first claim, suppose $S_1\cup S_2\subsetneq S$ and pick $x_0\in S\setminus(S_1\cup S_2)$. Also, pick $x_1\in S_1\setminus S_2$ and $x_2\in S_2\setminus S_1$. Let $x'$ be the point on the line segment between $x_0$ and $x_1$ such that $c_1^\top x'=c_{1,0}$. Similarly, let $x''$ be the point between $x_0$ and $x_2$ such that $c_2^\top x''=c_{2,0}$. Note that $x'\notin S_2$ and $x''\notin S_1$ by the convexity of $S\setminus S_1$ and $S\setminus S_2$. Then a point that is a strict convex combination of $x'$ and $x''$ is in $\conv(S_1\cup S_2)$ but not in $S_1\cup S_2$.

Corollary 9.1.2 in \cite{R1970} implies $S_1^+$ and $S_2^+$ are closed and $\rec S_1^+=\rec S_2^+=\rec S_1+\rec S_2$ because $S$ is pointed. The inclusions $S_1\subseteq S_1^+$ and $S_2\subseteq S_2^+$ imply that $\conv(S_1\cup S_2)\subseteq\conv(S_1^+\cup S_2^+)$. Furthermore, $\conv(S_1^+\cup S_2^+)$ is closed by Corollary 9.8.1 in \cite{R1970} since $S_1^+$ and $S_2^+$ have the same recession cone. Hence, $\clconv(S_1\cup S_2)\subseteq\conv(S_1^+\cup S_2^+)$. We claim $\clconv(S_1\cup S_2)=\conv(S_1^+\cup S_2^+)$. Let $x^+\in\conv(S_1^+\cup S_2^+)$. Then there exist $u_1\in S_1$, $v_2\in\rec S_2$, $u_2\in S_2$, and $v_1\in\rec S_1$ such that $x^+\in\conv\{u_1+v_2,u_2+v_1\}$. To prove the claim, it is enough to show that $u_1+v_2,u_2+v_1\in\clconv(S_1\cup S_2)$. Consider the point $u_1+v_2$ and the sequence
\begin{equation*}
\left\{\left(1-\frac{1}{k}\right)u_1+\frac{1}{k}\left(u_2+k v_2\right)\right\}_{k\in\N}.
\end{equation*}
For any $k\in\N$, we have $u_1\in S_1$ and $u_2+k v_2\in S_2$. Therefore, this sequence is in $\conv(S_1\cup S_2)$. Furthermore, it converges to $u_1+v_2$ as $k\rightarrow\infty$ which implies $u_1+v_2\in\clconv(S_1\cup S_2)$. A similar argument shows $u_2+v_1\in\clconv(S_1\cup S_2)$ and proves the claim.
\end{proof}

\subsection{Proofs of Section \ref{sec:SingleConvex}}

\begin{proof}[of Proposition~\ref{prop:ClosedSuff}]
Let $S_1^+:=S_1+\rec S_2$ and $S_2^+:=S_2+\rec S_1$. We have $\conv(S_1\cup S_2)\subseteq\clconv(S_1\cup S_2)=\conv(S_1^+\cup S_2^+)$ by Lemma~\ref{lem:closure}. We are going to show $\conv(S_1^+\cup S_2^+)\subseteq\conv(S_1\cup S_2)$ to prove that $\conv(S_1\cup S_2)$ is closed when \eqref{eq:ClosedSuff} is satisfied. Let $x^+\in S_1^+$. Then there exist $u_1\in S_1$ and $v_2\in\rec(S_2)$ such that $x^+=u_1+v_2$. If $c_2^\top v_2>0$, then there exists $\epsilon\geq 1$ such that $x^+ +\epsilon v_2\in S_2$ and we have $x^+\in\conv(S_1\cup S_2)$. Otherwise, $c_2^\top v_2=0$, and by the hypothesis, $c_1^\top v_2\geq 0$. This implies $x^+\in S_1$, and thus $S_1^+\subseteq\conv(S_1\cup S_2)$. Through a similar argument, one can show $S_2^+\subseteq\conv(S_1\cup S_2)$. Hence, $S_1^+\cup S_2^+\subseteq\conv(S_1\cup S_2)$. Taking the convex hull of both sides yields $\conv(S_1^+\cup S_2^+)\subseteq\conv(S_1\cup S_2)$.

For the converse, suppose condition (i) holds, and let $x^*\in S_1$ be such that $c_2^\top x^*\leq c_2^\top x$ for all $x\in S_1$. Note that $c_2^\top x^*<c_{2,0}$ since otherwise, $S_1\subseteq S_2$. Pick $\delta>0$ such that $x':=x^*+\delta r^*\notin S_1$. Then $x'\notin S_2$ too because $c_2^\top x'=c_2^\top x^*<c_{2,0}$. For any $0<\lambda<1$, $x_1\in S_1$, and $x_2\in S_2$, we can write $c_2^\top(\lambda x_1+(1-\lambda)x_2)\geq\lambda c_2^\top x^*+(1-\lambda)c_{2,0}>c_2^\top x'$. Hence, $x'\notin\conv(S_1\cup S_2)$. On the other hand, $x'\in S_1^+\subseteq\conv(S_1^+\cup S_2^+)=\clconv(S_1\cup S_2)$ where the last equality follows from Lemma~\ref{lem:closure}.
\end{proof}

\begin{proof}[of Corollary~\ref{cor:ClosedSuff}]
Suppose there exist $\beta_1,\beta_2\in\R$ such that $c_1-\beta_2c_2\in\K^*$ and $c_2-\beta_1c_1\in\K^*$. Consider the following minimization problem
\begin{equation*}
\inf_u\{c_1^\top u:\;c_2^\top u=0,u\in\K\}
\end{equation*}
and its dual
\begin{equation*}
\sup_\delta\{0:\;c_1-\delta c_2\in\K^*\}.
\end{equation*}
Because $\beta_2$ is a feasible solution to the dual problem, we have $c_1^\top u\geq 0$ for all $u\in\K$ such that $c_2^\top u=0$. Similarly, one can use the existence of $\beta_1$ to show that the second part of \eqref{eq:ClosedSuff} holds too. Then by Proposition~\ref{prop:ClosedSuff}, $\conv(C_1\cup C_2)$ is closed.
\end{proof}

\subsection{Proofs of Section \ref{sec:MultipleIneqs}}

\begin{proof}[of Proposition~\ref{prop:nontight}]
Every undominated valid inequality has to be tight on either $C_1$ or $C_2$; otherwise, we can just increase the right-hand side to obtain a dominating valid inequality. By Proposition~\ref{prop:VLIReduced}, undominated valid inequalities are of the form $\mu^\top x\geq c_{2,0}$  with $(\mu,\alpha_1,\alpha_2,\beta_1,\beta_2)$ satisfying \eqref{eq:VLIReduced}. In particular, we have $\beta_1,\beta_2>0$ such that $\min\{\beta_1c_{1,0},\beta_2c_{2,0}\}=c_{2,0}$. Now consider the following minimization problem
\begin{equation*}
\inf_u\{\mu^\top u:\;u\in C_1\}
\end{equation*}
and its dual
\begin{equation*}
\sup_\delta\{\delta c_{1,0}:\;\mu-\delta c_1\in\K^*,\;\delta\geq 0\}.
\end{equation*}
$C_1$ is a strictly feasible set by Assumption~\ref{As:A2}, so strong duality applies to this pair of problems and the dual problem admits an optimal solution $\delta^*\geq\beta_1>0$. Then
\begin{equation*}
\sign\{\delta^* c_{1,0}\}=\sign\{c_{1,0}\}=c_{1,0}>c_{2,0}.
\end{equation*}
Hence, the inequality $\mu^\top x\geq\mu_0$ cannot be tight on $C_1$.
\end{proof}

\begin{proof}[of Proposition~\ref{prop:TightOnOne}]
Let $\nu^\top x\geq c_{2,0}$ be a valid inequality of the form \eqref{eq:VLIReduced}. Then there exist $\alpha_1,\alpha_2,\beta_1,\beta_2$ such that $(\nu,\alpha_1,\alpha_2,\beta_1,\beta_2)$ satisfies \eqref{eq:VLIReduced}. In particular, $\min\{\beta_1c_{1,0},\beta_2c_{2,0}\}=c_{2,0}$. If $\beta_2c_{2,0}=c_{2,0}$, then $\beta_2=1$ and $\nu^\top x\geq c_{2,0}$ already has the desired form. Therefore, suppose $\beta_2c_{2,0}>c_{2,0}$. Then $c_{2,0}=\beta_1c_{1,0}=\beta_1c_{2,0}$ and thus $\beta_1=1$. We are going to show that $\nu^\top x\geq c_{2,0}$ is either dominated or has itself an equivalent representation \eqref{eq:VLIReduced} of the type claimed in the lemma.

Let $\delta\in\capD_2(c_1,c_2)$ and let $\gamma:=c_1-\delta c_2\in\K^*$. Then $\delta\geq 0$ because $\capD_1(c_1,c_2)=\emptyset$, and using Lemma~\ref{lem:inconsistent}, we have $\delta c_{2,0}<c_{1,0}=c_{2,0}$, which implies $\delta<1$. Then we can select $0<\lambda<1$ such that $\lambda\beta_2 c_{2,0}+(1-\lambda)\delta c_{2,0}=c_{2,0}$. Let us define $\alpha_1':=\lambda\alpha_1$, $\alpha_2':=\lambda\alpha_2+(1-\lambda)\gamma$, $\beta_2':=\lambda\beta_2+(1-\lambda)\delta=1$, and $\mu:=\lambda\nu+(1-\lambda)c_1=\lambda\alpha_1+c_1$. With these definitions, $\mu^\top x\geq c_{2,0}$ is a valid inequality for $\clconv(C_1\cup C_2)$ because $(\mu,c_{2,0},\alpha_1',\alpha_2',1,1)$ satisfies \eqref{eq:VLI}. Furthermore, $\nu-\mu=(1-\lambda)\alpha_1\in\K^*$. This shows that $\nu^\top x\geq c_{2,0}$ is dominated by $\mu^\top x\geq c_{2,0}$ if $\alpha_1\neq 0$ and has an equivalent representation \eqref{eq:VLI} with $\beta_2'=1$ if $\alpha_1=0$. In the first case, we are done. In the second case, we are done if $\alpha_2'\in\bd\K^*$. Otherwise, we can find a valid inequality that dominates $\nu^\top x\geq 1$ as in the proof of Proposition~\ref{prop:VLIReduced}.
\end{proof}

\subsection{Proofs of Section \ref{sec:p-order}}

\begin{proof}[of Corollary~\ref{cor:p-order}]
Let $q\in(1,\infty)$ be such that $\frac{1}{p}+\frac{1}{q}=1$. The sets $C_1$ and $C_2$ satisfy Assumptions~\ref{As:A1} and \ref{As:A2} because $e^i\notin\K_q^n$. Since we are considering a split disjunction, $\conv(C_1\cup C_2)$ is closed by Corollary~\ref{cor:ClosedSuff}. Using Propositions~\ref{prop:VLIReduced} and \ref{prop:closedness} and Lemma~\ref{lem:tight}, we see that any undominated valid linear inequality for $\conv(C_1\cup C_2)$ has the form $\mu^\top x\geq 1$ with $(\mu,\alpha_1,\alpha_2)$ satisfying
\begin{equation*}
\begin{gathered}
\mu=\alpha_1+t_1e^i,\\
\mu=\alpha_2-t_2e^i,\\
\alpha_1,\alpha_2\in\bd\K_q^n.
\end{gathered}
\end{equation*}
Let
\begin{equation*}
\capM\!:=\left\{\mu\in\R^n\!:\,\exists\alpha_1,\alpha_2\in\bd\K_q^n~\text{ s.t. }~\mu=\alpha_1+t_1e^i=\alpha_2-t_2e^i\right\}.
\end{equation*}
Then we can write
\begin{align*}
\capM=&\left\{\mu\in\R^n:\;\norm{\tilde{\mu}-t_1\tilde{e}^i}_q=\mu_n,\;
\norm{\tilde{\mu}+t_2\tilde{e}^i}_q=\mu_n\right\}\\
=&\left\{\mu\in\R^n:\;
\norm{\tilde{\mu}-t_1\tilde{e}^i}_q=\norm{\tilde{\mu}+t_2\tilde{e}^i}_q,\;
\norm{\tilde{\mu}+t_2\tilde{e}^i}_q=\mu_n
\right\}\\
=&\left\{\mu\in\R^n:\;2\mu_i=t_1-t_2,\;\norm{\tilde{\mu}+t_2\tilde{e}^i}_q=\mu_n\right\}.
\end{align*}
Therefore, we obtain
\begin{align*}
x&\in\conv(C_1\cup C_2)
\Leftrightarrow x\in\K_p^n\text{ and }\mu^\top x\geq 1\;~\forall\mu\in\capM.\\
&\Leftrightarrow x\in\K_2^n\text{ and }\inf_\mu\left\{\mu^\top x:\;\mu\in\capM(\beta,1)\right\}\geq 1.
\end{align*}
Note that the second equality constraint in the description of $\capM$ makes this optimization problem non-convex. However, we can relax this problematic constraint to an inequality without any loss of generality. In fact, consider the relaxation
\begin{equation*}
\inf_\mu\left\{\mu^\top x:\;2\mu_i=t_1-t_2,\;\norm{\tilde{\mu}+t_2\tilde{e}^i}_q\leq\mu_n\right\}.
\end{equation*}
The feasible region of this relaxation is the intersection of a hyperplane with a closed, convex cone shifted by the vector $t_2e^i$. Any solution which is feasible to the relaxation but not the original problem can be expressed as a convex combination of solutions feasible to the original problem. Because we are optimizing a linear function, this shows that the relaxation is equivalent to the original problem. Thus, we have
\begin{align*}
x&\in\conv(C_1\cup C_2)\\
&\Leftrightarrow x\in\K_p^n\text{ and }
\inf_\mu\left\{\mu^\top x:\;2\mu_i=t_1-t_2,\;\mu+t_2e^i\in\K_q^n\right\}\geq 1.
\end{align*}
The minimization problem in the last line above is strictly feasible since $e^n$ is a recession direction of the feasible region and belongs to $\intt\K_q^n$. Hence, its dual problem is solvable whenever it is feasible, strong duality applies, and we can replace the problem in the last line with its dual without any loss of generality. The dual problem is given by
\begin{align*}
&\sup_{\rho,\tau}\left\{-t_2(e^i)^\top\rho+(t_1-t_2)\tau:\;\rho+2\tau e^i=x,\;\rho\in\K_p^n\right\}\\
=&\sup_{\tau}\left\{-t_2x_i+(t_1+t_2)\tau:\;x-2\tau e^i\in\K_p^n\right\}
\end{align*}
and it is feasible for $x\in\K_p^n$ with $\tau=0$. Thus, we obtain
\begin{align*}
x\in&\conv(C_1\cup C_2)\\
&\Leftrightarrow x\in\K_p^n\text{ and }\sup_{\tau}\left\{-t_2x_i+(t_1+t_2)\tau:\;x-2\tau e^i\in\K_p^n\right\}\geq 1.
\end{align*}
Note that, for any given $x\in\K_p^n$, the problem above involves maximizing a linear function over a closed interval and the coefficient of $\tau$ in the objective function is positive; therefore, the optimum solution $\tau^*(x)$ will occur at the larger of the two endpoints of this interval which is
\begin{equation*}
\tau^*(x)=\frac{1}{2}x_i+\left(x_n^p-\sum_{j\notin\{i,n\}}|x_j|^p\right)^{\frac{1}{p}}.
\end{equation*}
Therefore,
\begin{align}
x\in&\conv(C_1\cup C_2)\nonumber\\
&\Leftrightarrow x\in\K_p^n\mbox{ and }-t_2x_i+(t_1+t_2)\tau^*(x)\geq 1.\nonumber\\
&\Leftrightarrow x\in\K_p^n\mbox{ and }\frac{2-(t_1-t_2)x_i}{t_1+t_2}\leq\left(x_n^p-\sum_{j\notin\{i,n\}}|x_j|^p \right)^{\frac{1}{p}}.\label{eq:pMain}
\end{align}

The validity and convexity of the above inequality follow from its derivation. Moreover, due to its derivation, this inequality precisely captures all of the undominated valid linear inequalities for $C_1\cup C_2$. Hence, together with the cone constraint $x\in\K_p^n$, the inequality \eqref{eq:pMain} is sufficient to describe $\conv(C_1\cup C_2)$.

In order to arrive at Corollary~\ref{cor:p-order}, we further claim that, for all $x\in\K_p^n$,
\begin{equation}\label{eq:pComplement}
\frac{-2+(t_1-t_2)x_i}{t_1+t_2}\leq\left(x_n^p-\sum_{j\notin\{i,n\}}|x_j|^p\right)^{\frac{1}{p}}.
\end{equation}
Let $u\in\K_p^n$. When $(t_1-t_2)u_i\leq 2$, the left-hand side of \eqref{eq:pComplement} is non-positive and the claim is satisfied trivially. Otherwise, $(t_1-t_2)u_i>2$ which implies that either $t_1 u_i>1$ ($u\in C_1$) or $-t_2 u_i>1$ ($u\in C_2$). Because $u\in\K_p^n$, we can write $\left(u_n^p-\sum_{j\notin\{i,n\}}|u_j|^p\right)^{\frac{1}{p}}\geq |u_i|$; therefore, all we need to show is that $|u_i|\geq\frac{(t_1-t_2)u_i-2}{t_1+t_2}$ or, equivalently, $(t_1+t_2)|u_i|\geq(t_1-t_2)u_i-2$. Suppose $u\in C_1$. Then we have $t_1u_i\geq 1$, and since $C_1\cap C_2=\emptyset$, we get $u\not\in C_2$, implying $-t_2u_i<1$. Therefore, $-2+(t_1-t_2)u_i<(2t_2u_i)+(t_1-t_2)u_i=(t_1+t_2)u_i\leq|(t_1+t_2)u_i|$ which proves the desired relation. Now suppose $u\in C_2$. Similarly, we obtain $t_1u_i<1$ which allows us to write $-2+(t_1-t_2)u_i<-2t_1u_i+(t_1-t_2)u_i=-(t_1+t_2)u_i\leq|(t_1+t_2)u_i|$ and completes the proof of the claim.

By combining the inequalities \eqref{eq:pMain} and \eqref{eq:pComplement}, we conclude that
\begin{equation*}
\frac{|2-(t_1-t_2)x_i|}{t_1+t_2}\leq\left(x_n^p-\sum_{j\notin\{i,n\}}|x_j|^p\right)^{\frac{1}{p}}.
\end{equation*}
Noting that $t_1+t_2>0$, taking the $p$-th power of both sides of the above inequality, and rearranging its terms, we arrive at
\begin{equation*}
\left(\left|(t_1-t_2)x_i-2\right|^p+( t_1+t_2)^p\sum_{j\notin\{i,n\}}|x_j|^p\right)^{\frac{1}{p}}\leq(t_1+t_2)x_n
\end{equation*}
which can be equivalently expressed as the desired conic valid inequality.
\end{proof}